\documentclass[a4paper,11pt]{article}
\usepackage[dvipsnames]{xcolor}
\usepackage{amsmath}
\usepackage{amsthm}
\usepackage{amssymb}
\usepackage{bbm}
\usepackage{enumitem}
\usepackage{epsfig}
\usepackage{graphics}
\usepackage{graphicx}
\usepackage{fancyhdr}
\usepackage[top = 2.5cm, bottom = 3.0cm, left = 2cm, right = 2cm]{geometry}

\usepackage{parskip}
\usepackage{cleveref}
\usepackage{xcolor}
\begingroup
    \makeatletter
    \@for\theoremstyle:=definition,remark,plain\do{%
        \expandafter\g@addto@macro\csname th@\theoremstyle\endcsname{%
            \addtolength\thm@preskip\parskip
            }%
        }
\endgroup
\usepackage{tikz}
\usetikzlibrary{automata,intersections,shapes,arrows,calc,positioning,decorations}
\usepackage{pgfplots}
\usepackage{tkz-fct}
\usepackage[dvipsnames]{xcolor}

\newcommand{\B}{\mathbb B}

\newcommand{\mR}{\mathbb R}
\newcommand{\R}{\mathbb R}

\newcommand{\N}{\mathbb N}

\newcommand{\sB}{{\mathcal B}}

\newcommand{\sE}{{\mathcal E}}

\newcommand{\sT}{{\mathcal T}}

\newcommand{\rd}{{\rm d}}

\numberwithin{equation}{section}
\numberwithin{figure}{section}
\numberwithin{table}{section}

\newtheorem{theorem}{Theorem}[section]
\newtheorem{proposition}[theorem]{Proposition}
\newtheorem{corollary}[theorem]{Corollary}
\newtheorem{lemma}[theorem]{Lemma}
\newtheorem{remark}[theorem]{Remark}
\newtheorem{example}[theorem]{Example}
\title{Well-posedness of Lur'e systems with feedthrough}
\author{Authors}
\author{Chris Guiver\thanks{School of Computing, Engineering \& the Built Environment, Edinburgh Napier University, Merchiston Campus, Edinburgh, UK, email: {\tt c.guiver@napier.ac.uk}, ORCID: 0000-0002-6881-7020}\,
\thanks{Corresponding author}
\and Hartmut Logemann\thanks{Department of Mathematical Sciences,
University of Bath, Bath, UK, email: {\tt h.logemann@bath.ac.uk}}
}
\date{Preprint uploaded to arxiv \\ August 2025}

\begin{document}
\maketitle

%
%
%
%
\textbf{Abstract.} For a large class of Lur'e systems with time-varying nonlinearities and feedthrough we consider several well-posedness issues, namely:
existence, continuation, blow-up in finite-time, forward completeness and uniqueness of solutions. Lur'e systems with feedthrough are systems of forced, nonlinear ordinary differential equations coupled with a nonlinear algebraic equation determining the output
of the system. The presence of feedthrough means that the algebraic equation is implicit in the output, and, in general, the output may not be expressible by an analytic formula in terms of the state and the input. Simple examples illustrate that the well-posedness properties of such systems are not necessarily guaranteed by assumptions sufficient for the corresponding well-posedness properties of Lur'e systems without feedthrough. We provide
sufficient conditions for the well-posedness properties mentioned above, using global inversion
theorems from real analysis and tools from non-smooth analysis and differential inclusions. The theory is illustrated with examples.

%
%
\textbf{Keywords.} 
Differential inclusions, 
existence of solutions, 
feedthrough, 
global inversion theorems, 
Lur'e systems, 
nonlinear differential equations

%
%
{\bfseries MSC (2020).} 34A12, 34A34, 34A60, 93B52, 93C10.

%
%
\section{Introduction}
Lur'e systems are a common and important class of nonlinear control systems. Their stability theory, often termed absolute stability, has generated significant interest, and dates
back to the work of Soviet scholars Lur'e (also written in English as Lurie or Lurye) and Postnikov in the 1940s. In fact, the Aizerman conjecture and the research on stability
conditions for Lur'e systems it triggered stand at the beginning of
nonlinear control theory, see \cite{l06} for historical notes.\footnote[4]{\,As is well known,  the Aizerman conjecture is not true in general. However, the problem of identifying classes of
Lur'e systems for which the conjecture is true continues to be of considerable interest,
see, for example, \cite{dgt23,nmig19}. We also mention in this context that the complex Aizerman conjecture is true, see \cite[Section 5.6.3]{hp05}.}

The literature on Lur'e systems
is vast and detailed treatments of Lur'e systems and absolute stability theory can be found in many places, including \cite{ag64,a13,blme20,dv75,g05,hc08,k02,nt73,v93}. Whilst absolute stability theory addresses global asymptotic stability of unforced Lur'e systems, more recently, input-to-state stability properties of forced Lur'e systems have been investigated, see \cite{at02,gl20,gl24,glo19,jlr09,jlr11,sl15}.
Furthermore, current interest in Lur'e systems is in part owing to their appearance in various neural network architectures and in mathematical biology, see, for instance, \cite{cgdrb23,dgt24,rwm20} and \cite{bglt16,fgl21,fglp19,st13}, respectively, and the references therein.

In this paper, we consider well-posedness properties of the following general class of controlled (or forced) Lur'e differential equations:
 \begin{subequations}\label{eq:lure_ss}
 \begin{align}
 \dot x(t) & = Ax(t)+ Bf(t,y(t)) + B_{\rm e}v(t),\quad x(t_0)=x^0,\,\,\,  t\geq t_0\geq 0,
 \label{eq:lure_ss_a}\\
 y(t) & = Cx(t)+Df(t,y(t))+ D_{\rm e}v(t),\label{eq:lure_ss_b}
 \end{align}
 \end{subequations}
where $v, x, y$ denote the control (external forcing or disturbance), state and output variables, respectively. The model data $A,B,B_{\rm e}, C, D$ and $D_{\rm e}$ are appropriately sized matrices, and $f$ is a time-dependent nonlinearity. It is assumed that $f$ is a Caratheodory
function, that is, $f(t,\,\cdot\,)$ is continuous and $f(\,\cdot\,,\xi)$ is measurable for
each fixed $t\geq 0$ and each fixed $\xi$, respectively.
System \eqref{eq:lure_ss} arises as the feedback connection of a linear control system and nonlinear, static, but possibly time-varying, output feedback $u(t)=
f(t,y(t))$ as illustrated in Figure \ref{fig:lure_system}. It is shown in Section
\ref{sec:feedthrough} that system \eqref{eq:lure_ss} is sufficiently general to capture a number of scenarios common in control theory and engineering.
\begin{figure}[h!]
\centering
\begin{tikzpicture}
\coordinate (O) at (0,0);
\node[draw, thick, minimum width=0.5cm, minimum height=0.85cm, anchor=south west, text width = 3.5cm, align = center] (P) at (O) {\footnotesize\begin{tabular}{c}
$\dot x = Ax + Bu + B_{\rm e} v$  \\
$ y = Cx + Du + D_{\rm e} v$
\end{tabular}};
\node[draw, thick, minimum width=0.5cm, minimum height=0.85cm, text width=1cm, align=center] (C) at ($(P.270) - (0,1)$) {$f$};
\draw[thick, -latex] ($(P.170) - (1,0)$) -- (P.170) node[above, pos=0.5] {\footnotesize{$v$}};
\draw[thick] (P.0) -- ($(P.0)+(0.5,0)$);
\draw[thick,-latex] ($(P.0)+(0.5,0)$) |- (C.0)node [right, pos=0.25]{\footnotesize{$y$}};


\draw[thick] (C.180)  -|  ($(P.190)-(0.5,0)$)node [left, pos=0.8]{\footnotesize{$u$}};
\draw[thick,-latex] ($(P.190)-(0.5,0)$) -- (P.190);
\end{tikzpicture}
\caption{Forced Lur'e system}
\label{fig:lure_system}
\end{figure}
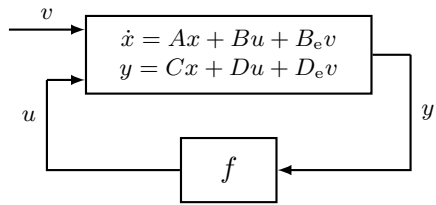

The matrix $D$ is called the feedthrough matrix, or simply, feedthrough.\footnote[2]{\,The matrix $D_{\rm e}$ also represents a feedthrough, an instantaneous impact of the input $v$ on the output $y$. However, it is the feedthrough $D$ which makes \eqref{eq:lure_ss_b} into an implicit equation for $y$, thereby potentially causing well-posedness issues. Therefore, in this paper,
the word ``feedthrough'' usually refers to $D$.}
Note that if $D=0$, then $y$ may be eliminated from \eqref{eq:lure_ss} leaving a system of nonlinear forced ordinary differential equations, the well-posedness of which is standard. However, if $D$ is non-zero, then equation \eqref{eq:lure_ss_b} is implicit in $y$ and this makes the analysis of well-posedness properties of \eqref{eq:lure_ss} more challenging.

Focussing on the case $D \neq 0$, we will discuss certain natural existence, uniqueness and continuation questions relating to system~\eqref{eq:lure_ss}:
\begin{itemize}
\item[(a)] Given~$v\in L^\infty_{\rm loc}(\R_+,\R^{m_{\rm e}})$, $t_0\geq 0$ and $x^0\in\R^n$, does there exist a pair $(x,y)$ satisfying \eqref{eq:lure_ss} such that
    $x(t_0)=x^0$?
\item[(b)] Under what conditions is uniqueness guaranteed?
\item[(c)] If $(v,x,y)$ satisfies \eqref{eq:lure_ss} and is maximally defined on $[t_0,\tau)$
with $t_0<\tau<\infty$, is it true that
\[ \limsup_{t\uparrow \tau}\|x(t)\|+\int_{t_0}^\tau\big(\|y(s)\|+\|f(s,y(s))\|\big){\rm d}s=\infty\;?\]
\item[(d)] Assuming that the nonlinearity $f$ satisfies $\|f(t,\xi)\|\leq a(t)+b\|\xi\|$ for all $t\geq 0$ and all $\xi$, where $a$ is locally integrable and
    $b$ is a positive constant: is it true that \eqref{eq:lure_ss} is forward complete in the sense that if $(v,x,y)$ satisfies \eqref{eq:lure_ss} and is maximally defined on $[t_0,\tau)$, then $\tau=\infty$?
\end{itemize}
Whilst Lur'e systems with feedthrough are considered throughout the literature (including
recent publications), see, for example, the research monograph \cite[Sections 3.13 and 3.14]{blme20}, the text books
\cite[Sections 5.7 and 5.8]{hc08}, \cite[Section 7.1]{k02} and \cite[Section 5.6]{v93}, and the papers \cite{bonp24,bovp22,gl24,hb93,tkp09,vdd19,zt02}, it seems that their well-posedness properties are rarely addressed in any detail. The exceptions which we are aware of include
\cite[Lemma 3.94]{blme20}, \cite[Assumption 1 and Appendix B]{vdd19} and \cite[Claim 1 and Appendix B]{zt02} which deal with certain well-posedness issues for some special classes of time-independent nonlinearities. The paper \cite{gl24} by
the authors contains a well-posedness result for the case of time-independent, continuously differentiable $f$, see \cite[Proposition 3.2]{gl24}. In the sections below, we will provide
a systematic analysis of the well-posedness issues (a)-(d) listed above in the context
of the general class of Lur'e systems \eqref{eq:lure_ss} with time-dependent nonlinearity,
including the
development of sufficient conditions guaranteeing well-posedness in the sense of (a), (b), (c)
and/or (d). In particular, the well-posedness results in \cite{blme20,gl24,vdd19,zt02} are consequences of the general theory developed in Section \ref{sec:wp} below, as is the well-known
result for the linear case $f(t,\xi) = K(t)\xi$.

In Subsection \ref{subsec:invertible},
we identify conditions on $D$ and $f$ ensuring that the map $F_t: \xi\mapsto \xi-Df(t,\xi)$ is
invertible for each $t\geq 0$ and the function $(t,\xi)\mapsto F_t^{-1}(\xi)$ is a Caratheodory
function and bounded on bounded sets, in which case the output $y$ can be eliminated from the
differential equation \eqref{eq:lure_ss_a}; the resulting system is then a ``standard'' nonlinear ordinary differential equation. The results in Subsection \ref{subsec:invertible} relate to existence and continuation of solutions
to the initial-value problem, forward completeness and uniqueness of solutions.  One technical challenge throughout Section \ref{sec:wp} is the accommodation of the time dependence of the nonlinearity $f$: this
is accomplished by the assumption that the radial unboundedness of $F_t$ (which is necessary
for the invertibility of $F_t$) is locally uniform in $t$. Apart from the theory of ordinary differential
equations, the main technical ingredients are global inverse function theorems (which, we note, play also an important role in other areas of systems and circuits, see, for instance, \cite{bl07,s80,s02,wd72}). By working with Clarke's set-valued generalized derivative \cite{c90,clsw98}, we are able to avoid continuous differentiability assumptions, and only assume that $f(t,\xi)$ satisfies a mild local Lipschitz condition with respect to $\xi$, thereby
providing a sufficiently general framework which allows for common non-smooth nonlinearities such as deadzone and saturation.

In Subsection \ref{subsec:not_invertible}, we consider the case in which $F_t$ is not necessarily invertible
for all $t\geq 0$. We do this, by recasting \eqref{eq:lure_ss_a} as a differential inclusion through a set-valued ``inverse'' $F_t^{-1}(\xi)$, where this symbol now denotes the fibre
of $\xi$ under $F_t$.  The results in Subsection \ref{subsec:not_invertible} provide sufficient conditions for the existence of solutions and forward completeness of \eqref{eq:lure_ss}. The main technical tools are results from non-smooth analysis such as existence of solutions to
differential inclusions and certain selection theorems (including Filippov's selection theorem),
see \cite{af09,clsw98,d92,v00,v04}.

The remainder of this work is organised as follows. Section \ref{sec:preliminaries} collects relevant preliminary material on global homeomorphisms, Clarke's
set-valued generalized derivative and set-valued functions. Section \ref{sec:feedthrough} contains
a detailed discussion of the class of Lur'e systems under consideration, introduces a number
of concepts associated with this class of systems, including trajectories, behaviours,
the blow-up property and forward completeness. The  main results of this paper can be found in Section \ref{sec:wp}. Section \ref{sec:radial_unboundedness} provides sufficient conditions for the locally uniform radial unboundeness property of the family of maps $F_t$. Examples illustrating the theory are included throughout the presentation. Several technical results are relegated to the Appendix.
\section{Preliminaries}\label{sec:preliminaries}
Most of the mathematical notation and terminology we use is standard, with only a few items mentioned explicitly.

In the following, we shall collect a number of preliminaries, including certain global inversion results for functions
$\R^n\to\R^n$ which will play a key role in the analysis of Lur'e systems with feedthrough
(as they do in other areas of systems and circuits, see, for example \cite{bl07,s80,s02,wd72}).
In this context, we recall Brouwer's invariance of domain theorem.
\begin{theorem}\label{thm:invariance_of_domain}
Let $U\subset\R^n$ be open and $g:U\to\R^n$ be continuous. If $g$ is injective, then $V:=g(U)$ is
open and $g$ is a homeomorphism between $U$ and $V$.
\end{theorem}
Proofs of the above result can be found in books on algebraic topology, see, for example
\cite[Theorem 6.10.7]{d18}. An alternative proof, which does not rely on singular homology theory, and avoids methods from algebraic topology beyond Brouwer's fixed point theorem, can be found in
\cite{k98}.\footnote[2]{\,We became aware of this reference through T.\ Tao's blog post {\it Brouwer's fixed point and invariance of domain theorems, and Hilbert's fifth problem},
https://terrytao.wordpress.com/2011/06/13/brouwers-fixed-point-and-invariance-of-domain-theorems-and-hilberts-fifth-problem/}

A function $g:\R^n\to\R^m$ is said to be {\it locally injective} if, for every $z\in\R^n$, there
exists a neighbourhood $U$ of $z$ such that $g|_U$ is injective. Similarly, we say that a continuous function $g:\R^n\to\R^n$ is a {\it local homeomorphism} if, for every $z\in\R^n$, there exist an open neighbourhood $U$ of $z$ such that $V:=g(U)$ is open and $g|_U$ is a homeomorphism between $U$ and $V$. It follows from \Cref{thm:invariance_of_domain} that a locally injective continuous function
$g:\R^n\to\R^n$ is a local homeomorphism. A continuous function $g:\R^n\to\R^n$ is said to be a {\it Lipschitz homeomorphism} if it is a homeomorphism and $g$ and $g^{-1}$ are locally Lipschitz.
Furthermore, $g:\R^n\to\R^m$ is said to be {\it radially
unbounded} if $\|g(z)\|\to\infty$, whenever $\|z\|\to\infty$. A continuous function $g$ is radially
unbounded if, and only if, for every compact set $K\subset\R^m$, the preimage $g^{-1}(K)$ is compact.
The  Fr{\'e}chet derivative of a function $g:\R^n\to\R^m$ at the point $z$ is denoted by $({\rm d}g)(z)$. Let $N_g$ denote the set of all points at which $g$ is not Fr{\'e}chet differentiable. If $g:\R^n\to\R^m$ is locally Lipschitz, then, by Rademacher's theorem (see, for example, \cite[Corollary 4.19, Chapter 3]{clsw98}), $({\rm d}g)(z)$ exists for almost every $z$, or, equivalently, $N_g$ has zero measure. For a locally Lipschitz function $g:\R^n\to\R^m$, Clarke's set-valued generalized derivative $({\rm d}^{\rm c}g)(z)$ of $g$ at $z$ is defined by
\begin{equation}\label{eq:clarke_derivative}
({\rm d}^{\rm c}g)(z)={\rm co}\{M\in\R^{m\times n}:\mbox{$\exists\,z_k\in\R^n\backslash N_g$ s.t.
$z_k\to z$ and $({\rm d}g)(z_k)\to M$ as $k\to\infty$}\},
\end{equation}
where ${\rm co}\,S$ denotes the convex hull of the set $S$, see \cite{c76,c90,clsw98}.
Obviously, $({\rm d}g)(z)\in ({\rm d}^{\rm c}g)(z)$ for $z\in\R^n\backslash N_g$.
If $g$ is continuously differentiable, then $({\rm d}^{\rm c}g)(z)=\{({\rm d}g)(z)\}$ for all $z\in\R^n$.

The following lemma will be used freely. Although the lemma is certainly known, we were not able to find a reference, and hence, we will give a proof.
\begin{lemma}\label{lem:Lg}
For locally Lipschitz $g:\R^n\to\R^m$ and $L\in\R^{q\times m}$, we have that
$({\rm d}^{\rm c}Lg)(z)=L({\rm d}^{\rm c}g)(z)$ for all $z\in\R^n$.
\end{lemma}
\begin{proof} Define a null set $N\subset\R^n$ by $N:=N_g\cup N_{Lg}$. For $z\in\R^n$ set
\[
S_g(z):=\{M\in\R^{m\times n}:\mbox{$\exists\,z_k\in\R^n\backslash N$ s.t.
$z_k\to z$ and $({\rm d}g)(z_k)\to M$ as $k\to\infty$}\}
\]
and
\[
S_{Lg}(z):=\{M\in\R^{q\times n}:\mbox{$\exists\,z_k\in\R^n\backslash N$ s.t.
$z_k\to z$ and $({\rm d}Lg)(z_k)\to M$ as $k\to\infty$}\}.
\]
It is well known that in the set on the right-hand side of \eqref{eq:clarke_derivative},
the defining equation of Clarke's set-valued generalized derivative,
an arbitrary null set $E\subset\R^n$ can be avoided, in the sense that only sequences
$(z_k)$ satisfying $z_k\in\R^n\backslash (N_g\cup E)$ are considered, without changing
the set-valued derivative, see \cite[Theorem 8.1, Chapter 2 and p.\ 133]{clsw98}. Consequently
\[
({\rm d}^{\rm c}g)(z)={\rm co}\,S_g(z)\quad\mbox{and}\quad({\rm d}^{\rm c}Lg)(z)={\rm co}\,S_{Lg}(z).
\]
Thus, it is sufficient to prove that $S_{Lg}(z)=LS_g(z)$ for all $z\in\R^n$. It is obvious that
$LS_g(z)\subset S_{Lg}(z)$. To prove the reverse inclusion, let $M\in S_{Lg}(z)$. Then there exists
a sequence $(z_k)$ in $\R^n\backslash N$ such that $\lim_{k\to\infty}z_k=z$ and
\[
\lim_{k\to\infty}L({\rm d}g)(z_k)=\lim_{k\to\infty}({\rm d}Lg)(z_k)=M.
\]
It follows from the Lipschitz property that the sequence $\big(({\rm d}g)(z_k)\big)_k$ is bounded. Hence, there exists a subsequence $\big(({\rm d}g)(z_{k_j})\big)_j$ and $G\in\R^{m\times n}$ such that
$({\rm d}g)(z_{k_j})\to G$ as $j\to\infty$, and so $G\in S_g(z)$. As $L({\rm d}g)(z_{k_j})\to M$
as $j\to\infty$, we conclude that $M=LG\in LS_g(z)$, and hence, $S_{Lg}(z)\subset LS_g(z)$, completing the proof.
\end{proof}
The open ball in $\R^n$ of radius $\rho>0$ and centred at $z$ is denoted by
$\B(z,\rho)$.
\begin{theorem}\label{thm:local_global}
The following statements hold.
\begin{enumerate}[label = {\bf (\arabic*)}, ref={\rm (\arabic*)}, leftmargin = 0ex, itemindent = 4.6ex]
\item\label{ls:local_global_s_1}
A continuous function $g:\R^n\to\R^n$ is a homeomorphism if, and only if, $g$ is locally injective and radially unbounded.
\item\label{ls:local_global_s_2}
A locally Lipschitz function $g:\R^n\to\R^n$ is a Lipschitz homeomorphism if, and only if, $g$ is radially unbounded, and, for every $z_0\in\R^n$, there exist $\delta,\varepsilon>0$ such that
\begin{equation}\label{eq:lower_lip}
\|g(z_1)-g(z_2)\|\geq\varepsilon\|z_1-z_2\|\quad\forall\,z_1,z_2\in\B(z_0,\delta).
\end{equation}
\item\label{ls:local_global_s_3}
If $g:\R^n\to\R^n$ is locally Lipschitz and radially unbounded and every matrix in $\bigcup_{z\in\R^n}({\rm d}^{\rm c}g)(z)$ is invertible, then $g$ is a Lipschitz homeomorphism.
\end{enumerate}
\end{theorem}
\begin{proof} \ref{ls:local_global_s_1} If $g$ is a homeomorphism, then it is injective and, {\it a fortiori}, locally injective. As the inverse function $g^{-1}$ is continuous, we have that $g^{-1}(K)$ is compact for every compact $K\subset\R^n$, and thus $g$ is radially unbounded.

Conversely, assume that $g$ is locally injective and radially unbounded. By \Cref{thm:invariance_of_domain}, $g$ is a local homeomorphism. It now follows
from \cite[Theorem 1.8, Chapter 3]{ap95} that $g$ is a homeomorphism.

\ref{ls:local_global_s_2} If $g$ is a Lipschitz homeomorphism, then, by statement \ref{ls:local_global_s_1}, $g$ is radially
unbounded. Furthermore, for $z_0\in\R^n$, there exist $\eta>0$ and $\lambda>0$ such that
$\|g^{-1}(\xi_1)-g^{-1}(\xi_2)\|\leq\lambda\|\xi_1-\xi_2\|$ for all $\xi_1,\xi_2\in\B(g(z_0),\eta)$.
Choose $\delta>0$ sufficiently small so that $g(z)\in\B(g(z_0),\eta)$ for all $z\in\B(z_0,\delta)$.
It follows that $\|g(z_1)-g(z_2)\|\geq(1/\lambda)\|z_1-z_2\|$ for all $z_1,z_2\in\B(z_0,\delta)$, whence \eqref{eq:lower_lip} holds with $\varepsilon=1/\lambda$.

Conversely, assume that $g$ is radially unbounded, and, for every $z_0\in\R^n$, there exist $\delta,\varepsilon>0$ such that \eqref{eq:lower_lip} is satisfied. The latter condition
implies that $g$ is locally injective. Hence, by statement \ref{ls:local_global_s_1}, $g$ is a homeomorphism. The local Lipschitz property of $g^{-1}$ follows from \eqref{eq:lower_lip}.

\ref{ls:local_global_s_3} It follows from Clarke's inversion theorem
for locally Lipschitz functions \cite{c76} (see also \cite[Theorem 7.1.1]{c90} or
\cite[Theorem 3.12, Chapter 3]{clsw98}) that $g$ is locally injective and every local inverse of $g$ is locally Lipschitz. The claim now follows from statement \ref{ls:local_global_s_1}.
\end{proof}
%
%
We continue with the statement and proof of a technical lemma which will be used in Sections \ref{sec:wp} and \ref{sec:radial_unboundedness}. The line segment determined by two points $z_1$ and $z_2$ in $\R^n$ is denoted by $[z_1,z_2]$, that is, $[z_1,z_2]:=\{z_1+s(z_2-z_1):0\leq s\leq 1\}$.
\begin{lemma}\label{lem:derivative_bounds}
Let $g:\R^n\to\R^m$ be locally Lipschitz, $U\subset\R^n$ be open, $U\not=\emptyset$, and
$N\subset\R^n$ be a null set such that $N_g\subset N$. The following
statements hold.
\begin{enumerate}[label = {\bf (\arabic*)}, ref={\rm (\arabic*)}, leftmargin = 0ex, itemindent = 4.6ex]
\item\label{ls:derivative_bounds_s_1}
If there exists $b>0$ such that $\|({\rm d}g)(z)\|\leq b$ for all $z\in U\backslash N$, then
\begin{equation}\label{eq:g_lipschitz}
\|g(z_1)-g(z_2)\|\leq b\|z_1-z_2\|\quad\mbox{for $z_1,z_2\in\R^n$ s.t. $[z_1,z_2]\subset U$}.
\end{equation}
\item\label{ls:derivative_bounds_s_2}
If $n=m$ and there exists $c\in\R$ such that $\langle ({\rm d}g)(z)\zeta,\zeta\rangle\leq c\|\zeta\|^2$ for all $\zeta\in\R^n$ and all $z\in U\backslash N$, then
\begin{equation}\label{eq:g_mono}
\langle g(z_1)-g(z_2),z_1-z_2\rangle\leq c\|z_1-z_2\|^2\quad\mbox{for $z_1,z_2\in\R^n$ s.t. $[z_1,z_2]\subset U$}.
\end{equation}
\end{enumerate}
\end{lemma}
\begin{proof}
In the first part of the proof we shall use an argument due to Clarke \cite{c76} (see also \cite[Proof of Lemma 2, p. 254]{c90} or \cite[Proof of Proposition 3.11, Chapter 3]{clsw98}).
Let $z_1,z_2\in U$ be such that $[z_1,z_2]\subset U$ and set $\xi:=z_2-z_1$. Let $H$ be the hyperplane in $\R^n$ perpendicular to $\xi$ and passing through $z_1$. Obviously, the set $U\cap N$ is of measure $0$. It follows from Fubini's theorem that, for almost every $y$ in $H\cap U$ (in the sense of $(n-1)$-dimensional Lebesgue
measure), the parametrization $l_y(s):=y+s\xi$, $0\leq s\leq 1$, of the line segment $[y,y+\xi]$ has the property that $E_y:=l_y^{-1}(N)\subset[0,1]$ is a null set, or, equivalently, the
intersection $[y,y+\xi]\cap N$ is a $1$-dimensional null set.
Let $Y\subset H\cap U$ be the set of all $y\in H\cap U$ such that $E_y$ is a null set and
$l_y(s)\in U$ for all $s\in [0,1]$. We note that $g$ is differentiable at $l_y(s)$ for all
$y\in Y$ and all $s\in[0,1]\backslash E_y$.
For $y\in Y$, the function $h_y:[0,1]\to\R^m$ defined by $h_y(s):=g(l_y(s))$
satisfies a Lipschitz condition on $[0,1]$, and thus, is absolutely continuous. Furthermore,
\begin{equation}\label{eq:h_properties}
h_y(1)-h_y(0)=\int_0^1h_y'(s){\rm d}s\qquad\mbox{and}\qquad h_y'(s)=({\rm d}g)(l_y(s))\xi\quad\forall\,y\in Y,\,\,\forall\,s\in[0,1]\backslash E_y.
\end{equation}
\ref{ls:derivative_bounds_s_1} It follows from \eqref{eq:h_properties} that
\[
\|g(y)-g(y+\xi))\|\leq b\|\xi\|=b\|z_1-z_2\|\quad\forall\,y\in Y.
\]
As $z_1\in\overline{Y}$, it follows from the continuity of $g$ that \eqref{eq:g_lipschitz} holds.

\ref{ls:derivative_bounds_s_2} We obtain from \eqref{eq:h_properties} that
\[
\langle g(y)-g(y+\xi),z_1-z_2\rangle\leq c\|\xi\|^2=c\|z_1-z_2\|^2\quad\forall\,y\in Y.
\]
Inequality \eqref{eq:g_mono} now follows from the continuity of $g$ and the fact that  $z_1\in\overline{Y}$.
\end{proof}
A function $g:\R_+\times\R^m\to\R^n$ is said to be a {\it Caratheodory function} if, for every $t\in \R_+$, the function $z\mapsto g(t,z)$ is continuous, and, for every $z\in\R^m$, the function $t\mapsto g(t,z)$ is measurable. It is well known that, if $g:\R_+\times\R^m\to\R^n$  is a {\it Caratheodory function} and $w:\R_+\to\R^m$ is measurable, then the function $t\mapsto g(t,w(t))$ is measurable,
see, for example, \cite[Lemma 8.2.3]{af09}.

Next, we review a number of concepts related to set-valued functions. Details can be found, for example in \cite{af09,d92} and \cite[Chapter 2]{v00}. For a set-valued map $G$ from ${\rm dom}(G)\subset\R^m$ to $2^{\R^n}$, the power set of $\R^n$, we shall use the notation $G:{\rm dom}(G)\rightrightarrows\R^n$ and $\Delta(G):=\{\xi\in{\rm dom}(G):G(\xi)\not=\emptyset\}$. We say that $G$ is {\it upper semicontinuous at} $\xi\in\Delta(G)$ if, for every open set $U\subset\R^n$ such that $G(\xi)\subset U$, there exists an open neighbourhood $V\subset\R^m$ of $\xi$ such that $G(\zeta)\subset U$ for all $\zeta\in V$. The map $G$ is called {\it upper semicontinuous on $S\subset\Delta(G)$} if it is upper semicontinuous at every point in $S$. We simply say that $G$ is {\it upper semicontinuous} if $G$ is upper semicontinuous on $\Delta(G)$. It is well known
that $G$ is upper semicontinuous if, and only if, for every closed $Y\subset\R^n$, there exists
a closed $X\subset\R^m$ such that $G^{-1}(Y)=X\cap\Delta(G)$, where
\[
G^{-1}(Y):=\{\xi\in{\rm dom}(G): G(\xi)\cap Y\not=\emptyset\}=\{\xi\in\Delta(G): G(\xi)\cap Y\not=\emptyset\},
\]
the preimage of $Y$ under $G$. A set-valued map $G:{\rm dom}(G)\rightrightarrows\R^n$ is said to be {\it measurable} if, for every open set $U\subset\R^n$, the set $G^{-1}(U)$ is Lebesgue measurable. A {\it measurable selection} of a set-valued map $G:{\rm dom}(G)\rightrightarrows\R^n$ is a measurable function $g:{\rm dom}(G)\to\R^n$ such that $g(\xi)\in G(\xi)$ for all $\xi\in{\rm dom}(G)$. Under the assumption that $\Delta(G)={\rm dom}(G)$, if a measurable function $\tilde g:{\rm dom}(G)\to\R^n$ satisfies $\tilde g(\xi)\in G(\xi)$ for almost every $\xi\in{\rm dom}(G)$, then there exists a measurable selection $g$ of $G$ such that $g(\xi)=\tilde g(\xi)$ for almost every $\xi\in{\rm dom}\,G$. Therefore, the function $\tilde g$ will also be referred to as a measurable selection of $G$.

Finally, the following convention applies to mixed ``almost everywhere'' and ``all'' quantifications: for $g:\R\times\R^m\to\R^n$, $T\subset\R$ and $S\subset\R^m$, the statement
\[
\mbox{$g(t,z)$ has property $P$}\quad\mbox{for a.e. $t\in T$ and all $z\in S$},
\]
should be understood as follows: there exists a set $E\subset T$ of measure zero such that
\[
\mbox{$g(t,z)$ has property $P$}\quad\forall\,(t,z)\in(T\backslash E)\times S,
\]
that is the ``a.e.'' part of the statement holds uniformly with respect to $z\in S$.
\section{Lur'e systems with feedthrough: definitions and concepts}\label{sec:feedthrough}
For fixed~$m,m_{\rm e},n,p\in \N$, let
 \[
(A,B,B_{\rm e},C,D,D_{\rm e})\in \mR^{n\times n}\times\mR^{n\times m}\times
 \mR^{n \times m_{\rm e}} \times \mR^{p\times n}
 \times \mR^{p\times m}\times \mR^{p\times m_{\rm e}}.
 \]
With the sextuple $(A,B,B_{\rm e},C,D,D_{\rm e})$, we associate the
following controlled and observed linear state-space system
 \begin{equation}\label{eq:lsss}
  \dot x = Ax + Bu + B_{\rm e}v, \quad y= Cx+Du + D_{\rm e}v.
 \end{equation}
Frequently, we will refer to~\eqref{eq:lsss} as the linear system $S:=(A,B,B_{\rm e},C,D,D_{\rm e})$.

Application of nonlinear output feedback of the form $u(t)=f(t,y(t))$ yields the closed-loop system \eqref{eq:lure_ss}, which will be denoted by $S^f:=(A,B,B_{\rm e},C,D,D_{\rm e},f)$.
For the following, it will be convenient to define $f\circ y$ by
\[
(f\circ y)(t):=f(t,y(t)).
\]
If~$f$ does not depend on~$t$, then~$f\circ y$ is simply the usual composition of the functions~$f$ and~$y$. The function~$v$ is an external input to the nonlinear feedback system~\eqref{eq:lure_ss}. It will always be assumed that $f$ is a Carathordory function.

Lur'e systems of the form~\eqref{eq:lure_ss} capture a large number of scenarios of interest, four of which are considered below.

{\it Scenario 1.} The system
\[
\dot x = Ax + B(f\circ y) + v_1, \quad y= Cx+D(f\circ y) + v_2
\]
is of the form~\eqref{eq:lure_ss} with~$m_{\rm e}=n+p$,~$B_{\rm e}=(I,0)$,~$D_{\rm e}=(0,I)$
and $v=(v_1^\top,v_2^\top)^\top$.

{\it Scenario 2.} Consider the following feedback scheme subject to output disturbances
\[
\dot x = Ax + Bu, \quad z= Cx+Du,\quad u=f\circ (z+d),
\]
where~$z$ is the undisturbed output and~$d$ is the output disturbance signal. Setting
$y:=z+d$, we have that
\[
\dot x = Ax + B(f\circ y), \quad y= Cx+D(f\circ y)+d,
\]
which is of the form~\eqref{eq:lure_ss} with~$m_{\rm e}=p$,~$B_{\rm e}=0$,~$D_{\rm e}=I$
and~$v=d$. Note that any boundedness properties of~$y$ can be used to infer boundedness properties of~$z$, provided suitable bounds on~$d$ are known.

{\it Scenario 3.} Here we consider the case of different nonlinearities in the state and output equations, namely
\[
\dot x = Ax + \tilde B(f_1\circ y)+B_{\rm e}v, \quad y= Cx+\tilde D(f_2\circ y)+D_{\rm e}v,
\]
where $\tilde B\in\R^{n\times m_1}$, $\tilde D\in\R^{p\times m_2}$, $f_1:\R_+\times\R^p\to\R^{m_1}$
and $f_2:\R_+\times\R^p\to\R^{m_2}$. The above system can be expressed in the form \eqref{eq:lure_ss}
with
\[
m:=m_1+m_2,\quad B:=(\tilde B,0),\quad D:=(0,\tilde D)\quad\mbox{and}\quad f:=\begin{pmatrix} f_1\\ f_2\end{pmatrix}.
\]
{\it Scenario 4.} In this scenario, we consider a so-called~$4$-block feedback scheme.
The underlying linear system is given by
\[
\dot x = Ax + Bu, \quad y= Cx+Du,
\]
with~$A\in\R^{n\times n}$,
\[
B=(B_1,B_2),\quad
C=\begin{pmatrix} C_1\\ C_2\end{pmatrix}\quad\mbox{and}\quad
D=\begin{pmatrix} D_{11} & D_{12}\\ D_{21} & D_{22}\end{pmatrix},
\]
where~$B_i\in\R^{n\times m_i}$,~$C_i\in\R^{p_i\times n}$ and~$D_{ij}\in\R^{p_i\times m_j}$,
$i,j=1,2$. Partitioning~$u$ and~$y$ accordingly,
\[
u=\begin{pmatrix} u_1\\ u_2\end{pmatrix},\,\,\, u_i(t)\in\R^{m_i},\quad
y=\begin{pmatrix} y_1\\ y_2\end{pmatrix},\,\,\, y_i(t)\in\R^{p_i},\quad i=1,2,
\]
consider the feedback system given by
\[
\dot x = Ax + Bu, \quad y= Cx+Du,\quad u_1=f\circ y_2,\quad u_2=v
\]
where~$f:\R_+\times\R^{p_2}\to\R^{m_1}$ and~$v$ is an external input. The system can be written as
\[
\dot x = Ax + B_1(f\circ y_2)+B_2v,\quad y_2= C_2x+D_{21}(f\circ y_2)+D_{22}v,\quad
y_1=C_1x+D_{11}u_1+D_{12}v.
\]
Note that the first two equations constitute a system of the form~\eqref{eq:lure_ss} with~$m=m_1$,~$m_{\rm e}=m_2$,~$p=p_2$,~$B_{\rm e}=B_2$, $C=C_2$, $D=D_{21}$ and~$D_{\rm e}=D_{22}$. Furthermore, we note that $y_1$ is completely determined by~$v$,~$u_1=f\circ y_1$ and~$x$, and hence by the first two equations (into which~$y_1$ does not enter).

The {\it behaviour}~$\sB(S)$ of the linear system~$S$ (or of~\eqref{eq:lsss}) is the linear subspace of all quadruples
\[
(u,v,x,y)\in L^1_{\rm loc}(\R_+,\R^m)\times L^\infty_{\rm loc}(\R_+,\R^{m_{\rm e}})\times W^{1,1}_{\rm loc}
(\R_+,\R^n)\times L^1_{\rm loc}(\R_+,\R^p)\,,
\]
which satisfy~\eqref{eq:lsss} for almost every~$t\geq 0$. The elements of~$\sB(S)$ are
called {\it trajectories} of~$S$.

Let $t_0\geq 0$. The {\it behaviour} $\sB(S^f,t_0)$ of $S^f$ (or of~\eqref{eq:lure_ss}) on
$[t_0,\infty)$ is the set of all triples
\[
(v,x,y)\in L^\infty_{\rm loc}(\R_+,\R^{m_{\rm e}})\times W^{1,1}_{\rm loc}([t_0,\infty),\R^n)\times L^1_{\rm loc}([t_0,\infty),\R^p)\,,
\]
such that $f\circ y\in L^1_{\rm loc}([t_0,\infty),\R^m)$ and $(v,x,y)$ satisfies
\eqref{eq:lure_ss} for almost every $t\geq t_0$. Elements in $\sB(S^f,t_0)$ will also be referred to as {\it trajectories} of $S^f$ on $[t_0,\infty)$ or trajectories of $S^f$ with initial time $t_0$. It is convenient to set  $\sB(S^f):= \sB(S^f,0)$, the elements of which will simply be referred to as {\it trajectories} of $S^f$.

In the following, if $t_0\geq 0$ and $z$ is a function defined on $[t_0,t_0+\tau)$, where $0<\tau\leq\infty$, we define $z^{t_0}$ on $[0,\tau)$ by $z^{t_0}(t)=z(t+t_0)$
for all $t\in[0,\tau)$. Similarly, we set $f^{t_0}(t,\xi):=f^{t_0}(t+t_0,\xi)$ for all $t\geq 0$
and all $\xi\in\R^p$. As $f$ is a Caratheodory function, so is $f^{t_0}$. We note that
\[
(v,x,y)\in\sB(S^f,t_0)\Longleftrightarrow (v^{t_0},x^{t_0},y^{t_0})\in\sB(S^{f^{t_0}})
\]
and
\[
(v,x,y)\in\sB(S^f,t_0)\Longleftrightarrow (f^{t_0}\circ y^{t_0},v^{t_0},x^{t_0},y^{t_0})\in\sB(S).
\]
If the nonlinearity $f$ does not depend on $t$, then the above equivalences remain valid when
$f^{t_0}$ is replaced by $f$.

For $t_0\geq 0$ and $t_0<\tau\leq\infty$, a triple~$(v,x,y)$ is said to be a {\it pre-trajectory} of~$S^f$ (or of \eqref{eq:lure_ss}) on $[t_0,\tau)$, if
\[
(v,x,y)\in L^\infty_{\rm loc}(\R_+,\R^{m_{\rm e}})\times W^{1,1}_{\rm loc}([t_0,\tau),\R^n)\times L^1_{\rm loc}([t_0,\tau),\R^p),
\]
$f\circ y\in L^1_{\rm loc}([t_0,\tau),\R^m)$ and~$(v,x,y)$ satisfies~\eqref{eq:lure_ss} for almost every~$t\in[t_0,\tau)$. If~$(v,x,y)$ is a pre-trajectory of~$S^f$ on~$[t_0,\tau)$, then we say that~$(v,x,y)$ is {\it maximally defined} if there does not exist a pre-trajectory
$(v,\tilde x,\tilde y)$ of~$S^f$ on~$[t_0,\tilde\tau)$ such that $\tilde\tau>\tau$ and
$(\tilde x,\tilde y)|_{[t_0,\tau)}=(x,y)$. The following result shows that every
pre-trajectory can be extended to a maximally defined pre-trajectory.
\begin{proposition}\label{pro:max_def}
Let~$(v,x,y)$ be a pre-trajectory of~\eqref{eq:lure_ss} on $[t_0,\tau)$, where
$0\leq t_0<\tau<\infty$. There exists a maximally defined pre-trajectory $(v,x_{\rm m},y_{\rm m})$ of \eqref{eq:lure_ss} on $[t_0,\tau_{\rm m})$ such that $\tau\leq\tau_{\rm m}\leq\infty$ and $(x_{\rm m},y_{\rm m})|_{[t_0,\tau)}=(x,y)$.
\end{proposition}
The proof is based on an application of Zorn's lemma and is a generalization of a similar argument familiar from the theory of ordinary differential equations. For the convenience of the reader, we have included the proof in the Appendix.

A maximally defined pre-trajectory on $[t_0,\tau)$ will also be referred as a
maximally defined pre-trajectory with initial time $t_0$.
Maximally defined pre-trajectories will play an important role in the the following and
therefore we define:
\[
\tilde\sB(S^f,t_0):=\{\mbox{maximally defined pre-trajectories of $S^f$ with initial time $t_0$}\}.
\]
It is convenient to set $\tilde\sB(S^f):=\tilde\sB(S^f,0)$. Obviously, $\sB(S^f,t_0)\subset\tilde\sB(S^f,t_0)$, and furthermore,
\[
(v,x,y)\in\tilde\sB(S^f,t_0)\Longleftrightarrow (v^{t_0},x^{t_0},y^{t_0})\in\tilde\sB(S^{f^{t_0}}).
\]
If, for every maximally defined pre-trajectory $(v,x,y)\in\tilde\sB(S^f,t_0)$ with bounded interval $[t_0,\tau)$ of definition, where $t_0<\tau<\infty$, it holds that
\[
\limsup_{t\uparrow\tau}\|x(t)\|+\int_{t_0}^\tau\big(\|y(s)\|+\|f(s,y(s))\|\big) \, \rd s=\infty,
\]
then~\eqref{eq:lure_ss} is said to have the {\it blow-up property}.

The following result is not surprising: it shows that the blow-up property holds, provided
that a suitable local existence assumption is satisfied.
\begin{proposition}\label{pro:blow_up}
Assume that, for all $t_0\geq 0$, all $x^0\in\R^n$ and all $v\in L^\infty_{\rm loc}(\R_+,\R^{m_{\rm e}})$, there exists $(v,x,y)\in\tilde\sB(S^f,t_0)$ such that $x(t_0)=x^0$.
Then, for a maximally defined pre-trajectory $(v,x,y)\in\tilde\sB(S^f,t_0)$ with bounded maximal interval of definition $[t_0,\tau)$ {\rm (}where $t_0<\tau<\infty${\rm)}, it holds that $\int_{t_0}^\tau\big(\|y(t)\|+\|Bf(t,y(t))\|\big){\rm d}t=\infty$. In
particular, \eqref{eq:lure_ss} has the blow-up property.
\end{proposition}
\begin{proof}  Let $(v,x,y)$ be a pre-trajectory of \eqref{eq:lure_ss} defined on the finite interval $[t_0,\tau)$ and assume that $\int_{t_0}^\tau\big(\|y(t)\|+\|Bf(t,y(t)\|\big){\rm d}t<\infty$. It is sufficient to show that~$(v,x,y)$ is not maximally defined.
As $\int_{t_0}^\tau\|Bf(t,y(t)\|{\rm d}t<\infty$, an application of the variation-of-parameters
formula to \eqref{eq:lure_ss_a} shows that the limit $\lim_{t\uparrow\tau}x(t)=:x(\tau)$ exists. It follows from the hypothesis that there exists a pre-trajectory $(v,\hat x,\hat y)\in\tilde\sB(S^f,\tau)$ defined on $[\tau,\theta)$ with $\tau<\theta\leq\infty$ and such that $\hat x(\tau)=x(\tau)$. Setting
\[
\tilde x(t):=\left\{\begin{aligned}
 & x(t), \quad t_0\leq t<\tau\\  & \hat x(t), \quad \tau\leq t<\theta
\end{aligned}\right.
\qquad\mbox{and}\qquad
\tilde y(t):=\left\{\begin{aligned}
 & y(t), \quad t_0\leq t<\tau\\  & \hat y(t), \quad \tau\leq t<\theta,
\end{aligned}\right.
\]
it is clear that $\tilde x\in W^{1,1}_{\rm loc}([t_0,\theta),\R^n)$,
$\tilde y\in L^1_{\rm loc}([t_0, \theta),\R^p)$, $f\circ \tilde y\in L^1_{\rm loc}([t_0,
\theta),\R^m)$ and
\[
\dot{\tilde x}(t)=A\tilde x(t)+Bf(t,\tilde y(t))+B_{\rm e}v(t),\quad
\tilde y(t)=C\tilde x(t)+Df(t,\tilde y(t))+D_{\rm e}v(t),\quad\mbox{for a.e.
$t\in[t_0,\theta)$.}
\]
Consequently, $(v,\tilde x,\tilde y)$ is a pre-trajectory of \eqref{eq:lure_ss} on $[\tau_0,\tau_0+\theta)$ which extends $(v,x,y)$, showing that $(v,x,y)$ is not maximally defined.
\end{proof}
Note that the hypothesis in Proposition \ref{pro:blow_up} holds if, and only if, for all $t_0\geq 0$, all $x^0\in\R^n$ and all $v\in L^\infty_{\rm loc}(\R_+,\R^{m_{\rm e}})$, there exists $(v,x,y)\in\tilde\sB(S^{f^{t_0}})$ such that $x(0)=x^0$. Furthermore, in the case wherein
$f(t,\xi)=f(\xi)$ (that is, $f$ does not depend on $t$), the hypothesis in Proposition \ref{pro:blow_up} is satisfied if, and only if, for all $x^0\in\R^n$ and all $v\in L^\infty_{\rm loc}(\R_+,\R^{m_{\rm e}})$, there exists $(v,x,y)\in\tilde\sB(S^f)$ such that $x(0)=x^0$.

For later purposes, we define the forward completeness and uniqueness properties.
System \eqref{eq:lure_ss} is said to be {\it forward complete} if $\tilde\sB(S^f,t_0)=
\sB(S^f,t_0)$ for every $t_0\geq 0$, that is, every maximally defined
pre-trajectory is a trajectory. Finally, we say that \eqref{eq:lure_ss}
has the {\it uniqueness property} if, for every~$v\in L^\infty_{\rm loc}(\R_+,\R^{m_{\rm e}})$ and every $t_0\geq 0$, any two pre-trajectories
$(v,x,y)$ and $(v,\tilde x, \tilde y)$ on $[t_0,\tau)$ and $[t_0,\tilde\tau)$,
respectively, where $t_0<\tau\leq\tilde\tau\leq\infty$, coincide on $[t_0,\tau)$ if $x(t_0)=\tilde x(t_0)$.

%
The following simple examples with time-independent globally Lipschitz $f$ show that if $D\not=0$, then
\begin{itemize}
\item for given $v\in L^\infty_{\rm loc}(\R_+,\R^{m_{\rm e}})$ and $x^0\in\R^n$, the set of pre-trajectories $(v,x,y)$ such that $x(t_0)=x^0$ may be empty;
\item the blow-up property may fail to hold;
\item for given $v\in L^\infty_{\rm loc}(\R_+,\R^{m_{\rm e}})$ and $x^0\in\R^n$, the set of pre-trajectories $(v,x,y)$ such that $x(t_0)=x^0$ may contain several elements despite $f$ being (globally) Lipschitz;
\item pre-trajectories may blow up in finite-time (implying that the system is not forward complete) despite $f$ being linearly bounded.
\end{itemize}
The first three of the four examples below are from \cite{gl24}, but as they are important in the context of the present paper, we repeat them here in a somewhat abbreviated form for the benefit of the reader. Unsurprisingly, the examples will show that lack of surjectivity or injectivity of the map $I-Df$ can cause problems regarding existence, uniqueness and blow-up.
\begin{example}\label{exa:ivp}
\begin{rm}
\begin{enumerate}[label = {\bf (\alph*)}, ref={\rm (\alph*)}, leftmargin = 0ex, itemindent = 4.6ex]
\item\label{ls:ivp_ex_a}
Consider~\eqref{eq:lure_ss} with
\[
A=\begin{pmatrix} 1 & 0\\ 0 & 0\end{pmatrix},\,\,\,
B=B_{\rm e}=\begin{pmatrix} 0 \\ 1 \end{pmatrix},\,\,\, C=(1,0),\,\,\,D=D_{\rm e}=1,
\]
and nonlinearity~$f$ given by
\[
f(\xi)=\left\{ \begin{aligned} & \xi+1, & & \xi<-2, \\
& \xi/2, & &-2\leq\xi\leq 2, \\
& \xi-1, & & \xi>2.
\end{aligned}\right.
\]
Let~$x^0=(a,0)^\top$,~$a\in\R$, and assume that~$(0,x,y)$ is a  pre-trajectory defined on some interval~$[0,\tau)$ and such that~$x(0)=x^0$. Noting that $Ce^{At}B\equiv 0$ and using the variation-of-parameters formula, it follows from \eqref{eq:lure_ss} that
\begin{equation}\label{eq:ivp_ex}
y(t)-f(y(t))=e^{t}a,\quad\forall\:t\in[0,\tau).
\end{equation}
As
\[
\xi-Df(\xi)=\xi-f(\xi)=\left\{ \begin{aligned} & -1, & & \xi<-2, \\
& \xi/2, & &-2\leq\xi\leq 2, \\
& 1, & & \xi>2,
\end{aligned}\right.
\]
we see that \eqref{eq:ivp_ex} does not have a solution
for any~$t\geq 0$ if~$|a|>1$. Hence, there does not exist any pre-trajectory~$(0,x,y)$
such that~$x(0)=(a,0)^\top$ if~$|a|>1$. We note that $I-Df$ is not surjective.
\item\label{ls:ivp_ex_b} We consider the example introduced in part~\ref{ls:ivp_ex_a}, now with~$a=1/2$. In this case, we see that~\eqref{eq:ivp_ex}, has the unique solution~$y(t)=e^t$ for every~$t\in[0,\ln 2)$ and does not have a solution whenever $t>\ln 2$.
Setting~$y(t):=e^t$ for all~$t\in[0,\ln 2)$ and
\[
x(t):=e^{At}x^0+\int_0^te^{A(t-s)}Bf(y(s)){\rm d}s=\begin{pmatrix} e^t/2\\ (e^t-1)/2\end{pmatrix},\quad\forall\:t\in[0,\ln 2),
\]
we conclude that~$(0,x,y)$ is a maximally defined pre-trajectory satisfying
$x(0)=(1/2,0)^\top$. It is clear that
\[
\limsup_{t\to\ln 2}\|x(t)\|+\int_0^{\ln 2}(|y(s)|+|f(y(s))|)\, {\rm d}s <\infty,
\]
showing that the system does not have the blow-up property.
\item\label{ls:ivp_ex_c}
Consider the scalar system
\[
\dot x=-x+f\circ y+v,\quad y=x+f\circ y+v,
\]
which is~\eqref{eq:lure_ss} with~$A=-1$ and~$B=B_{\rm e}=C=D=D_{\rm e}=1$. Consider the nonlinearity~$f$ given by
\[
f(\xi)=\left\{ \begin{aligned} & -3/4, & & \xi<-1/2, \\
& \xi(1-\xi), & &-1/2\leq \xi\leq 1/2, \\
& 1/4, & & \xi>1/2.
\end{aligned}\right.
\]
It is clear that $I-Df$ is not injective. Let~$x_1(t)\equiv 1/4$,~$y_1(t)\equiv 1/2$ and
\[
x_2(t)=\left\{ \begin{aligned} & (e^{-t}-1/2)^2, & & 0\leq t\leq\ln 2, \\
& 0, & & t>\ln 2, \end{aligned} \right.\qquad
y_2(t)=-\sqrt{x_2(t)}=\left\{ \begin{aligned} & 1/2-e^{-t}, & & 0\leq t\leq\ln 2, \\
& 0, & & t>\ln 2\,. \end{aligned} \right.
\]
It is easy to check that~$(0,x_1,y_1), (0,x_2,y_2)\in\sB(S^f)$. Since $x_1(0)=1/4=
x_2(0)$, we see that the system does not have the uniqueness property. Obviously, $f$ is globally Lipschitz, and thus, we conclude that the non-uniqueness is caused by the feedthrough and not by an absence of the Lipschitz property.
\item\label{ls:ivp_ex_d}
Consider \eqref{eq:lure_ss} with
\[
A=\begin{pmatrix} 1 & -1\\ -1 & 1\end{pmatrix},\,\,\,
B=B_{\rm e}=\begin{pmatrix} 1 \\ -1 \end{pmatrix},\,\,\, C=(1,1),\,\,\,D=D_{\rm e}=1,
\]
nonlinearity $f(\xi)=\xi-\arctan\xi$ and input $v(t)=t$. Set
$y(t):=\tan t$ for $t\in[0,\pi/2)$ and define
\[
x(t):=e^{2t}B\left(\int_0^te^{-2s}\tan s\,{\rm d}s-1\right),\quad\forall\,t\in[0,\pi/2).
\]
Routine calculations show that $(v,x,y)$ is a pre-trajectory of \eqref{eq:lure_ss} on $[0,\pi/2)$.
Since
\[
\int_0^t y(s){\rm d}s=-\ln(\cos t)\to\infty\quad\mbox{and}\quad \|x(t)\|\to\infty
\quad\mbox{as $t\uparrow\pi/2$},
\]
we see that the pre-trajectory $(v,x,y)$ blows up in finite time, and thus, the system is not forward complete. As $f$ is linearly bounded, we conclude that the lack of forward completeness is caused by the feedthrough and not by superlinear growth of the nonlinearity.
\hfill$\Diamond$
\end{enumerate}
\end{rm}
\end{example}
\section{Lur'e systems with feedthrough: well-posedness results}\label{sec:wp}
The map $F:\R_+\times\R^p\to\R^p$ given by
\[
F(t,\xi):=\xi-Df(t,\xi)\quad\forall\,(t,\xi)\in\R_+\times \R^p
\]
and certain of its properties will play a pivotal role in the rest of the paper.
Recall that $f$ is always assumed to be a Caratheodory function, and so $F$ is
a Caratheodory function.

For each $t\geq 0$, we set $f_t(\xi):=f(t,\xi)$ and $F_t(\xi):=F(t,\xi)$ and note that $f_t$ and $F_t$ are continuous maps from $\R^p$ to $\R^m$ and $\R^p$, respectively. In the case wherein $F_t$ is invertible for every $t \geq 0$, the variable $y$ can be eliminated from \eqref{eq:lure_ss} which, at least formally, yields an ordinary
differential equation. If $F_t$ is not invertible for all $t\geq 0$, then the output $y$ cannot be analytically expressed in terms of $x$ and $v$, and so, in general, \eqref{eq:lure_ss} is not a ``standard'' differential equation. We start by considering the case wherein $F_t$ is invertible for all $t\geq 0$.
\subsection{$F_t$ invertible for every $t$}\label{subsec:invertible}
In the following, we will say that $F_t$ is radially unbounded, {\it locally uniformly in $t$}, if, for all $\rho>0$ and all compact $T\subset\R_+$, there exists $\sigma\geq 0$ such that $\|F_t(\xi)\|\geq\rho$ for all $t\in T$ and all $\xi$ such that $\|\xi\|\geq\sigma$.
\begin{theorem}\label{thm:single_valued}
Assume that $F_t$ is locally injective for every $t\geq 0$ and that
$F_t$ is radially unbounded, locally uniformly in $t$. Furthermore,
assume that, for every compact subset $K\subset\R^p$, there exists
$\varphi\in L^1_{\rm loc}(\R_+,\R_+)$ such that
\begin{equation}\label{eq:single_valued}
\sup_{\xi\in K}\|f(t,\xi)\|\leq\varphi(t)\quad\forall\,t\geq 0.
\end{equation}
Let $v\in L^\infty_{\rm loc}(\R_+,\R^{m_{\rm e}})$, $t_0\geq 0$ and $x^0\in\R^n$. The following statements hold.
\begin{enumerate}[label = {\bf (\arabic*)}, ref={\rm (\arabic*)}, leftmargin = 0ex, itemindent = 4.6ex]
\item\label{ls:single_valued_s_1}
There exists $(v,x,y)\in\tilde\sB(S^f,t_0)$ such that $x(t_0)=x^0$.
\item\label{ls:single_valued_s_2}
If $(v,x,y)\in\tilde\sB(S^f,t_0)$ with bounded maximal interval of definition $[t_0,\tau)$,
where $t_0<\tau<\infty$, then $\int_{t_0}^\tau\|Bf(t,y(t))\|{\rm d}t=\infty$, $y\not\in L^\infty([0,\tau),\R^p)$ and $\|x(t)\|\to\infty$ as $t\uparrow\tau$. In particular,
system \eqref{eq:lure_ss} has the blow-up property.
\end{enumerate}
\end{theorem}
Note that local injectivity of $F_t$ together with radial unboundedness of $F_t$ is equivalent to
$F_t$ being a homeomorphism, see statement \ref{ls:local_global_s_1} of \Cref{thm:local_global}.
\begin{remark}\label{rem:a.e.}
\begin{rm}
The above result remains valid under slightly weaker assumptions on $F$. Assume that \eqref{eq:single_valued} holds, $F_t$ is locally injective for almost every $t\geq 0$
and $F_t$ is radially unbounded, {\it locally essentially uniformly in $t$}, in the sense
that, for all $\rho>0$ and all compact $T\subset\R_+$, there exists $\sigma\geq 0$ such that $\|F_t(\xi)\|\geq\rho$ for almost every $t\in T$ and all $\xi$ with $\|\xi\|\geq\sigma$.
We claim that the conclusions of \Cref{thm:single_valued}
are still valid. To see this, note that there exists a null set $E\subset\R_+$ such that
$F_t$ is locally injective for every $t\in\R_+\backslash E$, and, for all $\rho>0$ and all compact $T\subset\R_+$, there exist $\sigma\geq 0$ such that $\|F_t(\xi)\|\geq\rho$ for all $t\in T\backslash E$ and all $\xi$ with $\|\xi\|\geq\sigma$. Defining
\[
\hat f(t,\xi):=\left\{ \begin{aligned} & f(t,\xi), & & (t,\xi)\in(T\backslash E)\times\R^p, \\
& 0, & & (t,\xi)\in E\times\R^p
\end{aligned}\right.\qquad\mbox{and}\qquad
\hat F_t(\xi):=\hat F(t,\xi):=\xi-D\hat f(t,\xi),
\]
it is clear that $\hat f$ is a Caratheodory function, $\tilde\sB(S^{\hat f},t_0)=
\tilde\sB(S^f,t_0)$, $\hat F_t$ is locally injective
for every $t\geq 0$ and $\hat F_t$ is radially unbounded, locally uniformly in $t$. Hence, the above claim follows from an application of \Cref{thm:single_valued} to system \eqref{eq:lure_ss} with $f$ replaced by $\hat f$.
\hfill$\Diamond$
\end{rm}
\end{remark}
Whilst \Cref{thm:single_valued} is a special case of the more general \Cref{thm:inclusion_wp} below,  we provide a direct proof (not based on \Cref{thm:inclusion_wp}), because we feel it
is of interest to see how \eqref{eq:lure_ss} can be written in the form of a ``standard'' differential equation and how statements \ref{ls:single_valued_s_1} and \ref{ls:single_valued_s_2} then follow from the theory of ordinary differential equations.

{\it Proof of \Cref{thm:single_valued}.} Let $v\in L^\infty_{\rm loc}(\R_+,\R^{m_{\rm e}})$, $t_0\geq 0$ and $x^0\in\R^n$. Invoking \Cref{thm:local_global}, we see that the hypotheses on $F_t$ guarantee that $F_t$ is homeomorphism for all $t\geq 0$. Hence, for arbitrary, $\xi\in\R^p$,
we have that $\xi\in F_t(\R^p)=F(t,\R^p)$ for every $t\geq 0$, and it follows from Filippov's selection theorem (see, for example, \cite[Theorem 2.3.13]{v00})
that there exists a measurable function $w:\R_+\to\R^p$ such that $\xi=F(t,w(t))=F_t(w(t))$ for all $t\geq 0$. Hence, $F_t^{-1}(\xi)=w(t)$ for all $t\geq 0$, showing that the function $g:\R_+\times\R^p\to\R^p$ defined by $g(t,\xi):=F_t^{-1}(\xi)$ is a Caratheodory function. Eliminating $y$ from \eqref{eq:lure_ss} leads to
\begin{equation}\label{eq:lure_ss'}
\dot x(t)=Ax(t)+Bh(t,x(t))+B_{\rm e}v(t),\quad\mbox{where \,\,$h(t,z):=f\big(t,g(t,Cz+D_{\rm e}v(t))\big)$}.
\end{equation}
As $f$ and $g$ are Caratheodory functions, it follows that $h$ is also a Caratheodory function.
Furthermore, let $T\subset\R_+$ and $\Gamma\subset\R^n$ be compact and choose $\rho>0$ such that $Cz+D_{\rm e}v(t)\in\overline{\B(0,\rho)}$ for all $z\in\Gamma$ and for almost every $t\in T$. As $F_t$
is radially unbounded, uniformly for $t\in T$, the set $F_t^{-1}(\overline{\B(0,\rho)})$
is bounded, uniformly for $t\in T$, implying that there exists a compact set $K\subset\R^p$ such that
$g(t,\overline{\B(0,\rho)})=F_t^{-1}(\overline{\B(0,\rho)})\subset K$ for all $t\in T$. It follows now from \eqref{eq:single_valued} that there exists $\varphi\in L^1(T,\R_+)$ such that $\sup_{z\in\Gamma}\|h(t,z)\|\leq\varphi(t)$ for almost every $t\in T$. Hence, $h$ satisfies
the conditions required to apply well known existence and continuation results from the theory of ordinary differential equations to system \eqref{eq:lure_ss'}. Specifically,
statements \ref{ls:single_valued_s_1} and \ref{ls:single_valued_s_2}
follow from \cite[Chapter 2, Sec.\,1]{cl55} or \cite[Chapter III, \S 10, Supp.\,II]{w98}
(statement \ref{ls:single_valued_s_1} follows also from \cite[Section 7.4]{v04}).
\hfill$\Box$

The following corollary provides a condition ensuring forward completeness and existence
of trajectories for every initial condition.
\begin{corollary}\label{cor:single_valued}
Assume that $F_t$ is locally injective for every $t\geq 0$, for every compact subset $K\subset\R^p$, there exists $\varphi\in L^1_{\rm loc}(\R_+,\R_+)$ such that \eqref{eq:single_valued} holds, and,
for all compact sets $T\subset\R_+$, there exist $\rho>0$ and $c>0$ such that $c\|D\|<1$ and
\begin{equation}\label{eq:al_bound}
\sup_{\|\xi\|\geq\rho}\frac{\|f(t,\xi)\|}{\|\xi\|}\leq c\qquad\mbox{for a.e. $t\in T$}.
\end{equation}
Then $\tilde\sB(S^f,t_0)=\sB(S^f,t_0)$ for every $t_0\geq 0$ {\rm (}that is, \eqref{eq:lure_ss} is forward complete{\rm )}. Moreover, for all $v\in L^\infty_{\rm loc}(\R_+,\R^{m_{\rm e}})$, $t_0\geq 0$ and $x^0\in\R^n$, there exists $(v,x,y)\in\sB(S^f,t_0)$ such that $x(t_0)=x^0$.
\end{corollary}
\begin{proof} As $c\|D\|<1$, it follows from \eqref{eq:al_bound} that $F_t$ is radially unbounded, locally essentially uniformly in $t$ in the sense of \Cref{rem:a.e.}. Consequently, by \Cref{rem:a.e.}, the conclusions of \Cref{thm:single_valued} hold and it is sufficient to show that $\tilde\sB(S^f,t_0)=\sB(S^f,t_0)$. To this end, define $g$ and $h$ as in
the proof of \Cref{thm:single_valued}, that is, $g(t,\xi)=F_t^{-1}(\xi)$ and $h$ is defined as in \eqref{eq:lure_ss'}. As $F_t(g(t,\xi))=\xi$, it follows that
\begin{equation}\label{eq:g_ineq}
\|g(t,\xi)\|-\|D\|\|f(t,g(t,\xi))\|\leq\|\xi\|\quad\forall\,\xi\in\R^p,\,\,\forall\,t\geq 0.
\end{equation}
Let $\tau\in(t_0,\infty)$. There exist $\rho>0$ and $c>0$ such that $c\|D\|<1$ and
$\|f(t,\xi)\|\leq c\|\xi\|$ for all $\xi\in\R^p$ with $\|\xi\|\geq\rho$
and for almost every $t\in [t_0,\tau]$. Hence, by \eqref{eq:g_ineq}, the following implication
holds for almost every $t\in[t_0,\tau]$:
\[
\|g(t,\xi)\|\geq\rho\quad\Longrightarrow\quad\|g(t,\xi)\|\leq c_1\|\xi\|,\qquad
\mbox{where $c_1:=1/(1-c\|D\|)$}.
\]
Invoking \eqref{eq:single_valued} with $K=\overline{\B(0,\rho)}$, we see that there exists $\varphi\in L^1([t_0,\tau],\R_+)$ such that
\[
\|f(t,g(t,\xi))\|\leq\varphi(t)+cc_1\|\xi\|\quad\forall\,\xi\in\R^p,\,\,\,
\mbox{for a.e. $t\in[t_0,\tau]$.}
\]
This in turn implies that there exist $\varphi_1\in L^1([t_0,\tau],\R_+)$ and $c_2>0$ such that
$\|h(t,z)\|\leq\varphi_1(t)+c_2\|z\|$ for all $z\in\R^p$ and almost every $t\in [t_0,\tau]$.
It follows now from an application of \cite[Proposition 2.1.19]{hp05} to \eqref{eq:lure_ss'}
that $\tilde\sB(S^f,t_0)=\sB(S^f,t_0)$.
\end{proof}

The next result provides a sufficient condition guaranteeing existence for all initial conditions
and the blow-up and uniqueness properties.
\begin{theorem}\label{thm:single_valued_uniqueness}
Assume that there exists $\xi_0\in\R^p$ such that the function
$t\mapsto f(t,\xi_0)$ is locally integrable and $F_t$ is radially unbounded, locally
uniformly in $t$. Furthermore, assume that, for all compact sets $T\subset\R_+$ and
$K\subset\R^p$, there exists $\lambda\in L^1_{\rm loc}(\R_+,\R_+)$ and $\varepsilon>0$ such that
\begin{equation}\label{eq:lipschitz}
\|f(t,\xi)-f(t,\zeta)\|\leq\lambda(t)\|\xi-\zeta\|\quad\forall\,\xi,\zeta\in K,\,\,\forall\,t\geq 0
\end{equation}
and
\begin{equation}\label{eq:lower_lipschitz}
\|F(t,\xi)-F(t,\zeta)\|\geq\varepsilon\|\xi-\zeta\|\quad\forall\,\xi,\zeta\in K,\,\,\forall\,t\in T.
\end{equation}
Then statements \ref{ls:single_valued_s_1} and \ref{ls:single_valued_s_2} of \Cref{thm:single_valued} hold
and \eqref{eq:lure_ss} has the uniqueness property.
\end{theorem}
The remark below contains commentary on the conditions imposed in \Cref{thm:single_valued_uniqueness}.
\begin{remark}\label{rem:uniqueness'}
\begin{enumerate}[label = {\bf (\alph*)}, ref={\rm (\alph*)}, leftmargin = 0ex, itemindent = 4.6ex]
\begin{rm}
\item\label{ls:uniqueness_s_1} In the case that $f$ is linear and time-varying, that is $f(t,\xi) = K(t) \xi$ for measurable $K:\R_+\to\R^{m \times p}$, the conditions of
    \Cref{thm:single_valued_uniqueness} are satisfied if, and only if, $K$ is locally integrable and, for every compact subset $T\subset\R_+$, there exists $\varepsilon > 0$ such that $\min_{\| \xi \|=1} \|(I-DK(t))\xi\|\geq\varepsilon$ for all $t\in T$. In the time-independent case $K(t)\equiv K$, these conditions reduce to the invertibility of $I-DK$ (the latter being being the familiar well-posedness criterion for linear systems with feedthrough under linear output feedback with constant gain matrix, see, for example,
    \cite[Lemma 5.1, p.67]{zd98}).
\item\label{ls:uniqueness_s_2}
If, for all compact sets $T\subset\R_+$ and $K\subset\R^p$, there exist $\gamma_1\in(0,1)$ or $\gamma_2>1$ such that
\[
\|Df(t,\xi)-Df(t,\zeta)\|\leq\gamma_1\|\xi-\zeta\|\quad\forall\,\xi,\zeta\in K,\,\,
\forall\,t\in T,
\]
or
\[
\|Df(t,\xi)-Df(t,\zeta)\|\geq\gamma_2\|\xi-\zeta\|\quad\forall\,\xi,\zeta\in K,\,\,
\forall\,t\in T,
\]
then \eqref{eq:lower_lipschitz} holds with $\varepsilon=1-\gamma_1$ or $\varepsilon=\gamma_2-1$, respectively. Moreover, if the function $\lambda$ in \eqref{eq:lipschitz} is constant and $\lambda\|D\|<1$, then \eqref{eq:lower_lipschitz} is implied by \eqref{eq:lipschitz}.
\item\label{ls:uniqueness_s_3}
We remark that if $f$ is bounded on compact sets, then \Cref{thm:single_valued_uniqueness} remains valid if the Lipschitz condition
\eqref{eq:lipschitz} is replaced by the following, more localized, condition: for all $(\tau,\zeta)\in \R_+\times\R^p$, there exist positive constants $\lambda$, $\theta$ and $\rho$ such that $\|f(t,\xi_1)-f(t,\xi_2)\|\leq\lambda\|\xi_1-\xi_2\|$ for all $t\in(\tau-\theta,\tau+\theta)\cap\R_+$ and all $\xi_1,\xi_2\in\B(\zeta,\rho)$.
\hfill$\Diamond$
\end{rm}
\end{enumerate}
\end{remark}
The following corollary of \Cref{thm:single_valued_uniqueness} considers the
case wherein the maps $Df_t-\gamma I$ satisfy certain monotonicity conditions
for suitable $\gamma\in\R$.
\begin{corollary}\label{cor:dissipativity}
Assume that there exists $\xi_0\in\R^p$ such that the function
$t\mapsto f(t,\xi_0)$ is locally bounded, $f$ is locally Lipschitz in the
sense of \eqref{eq:lipschitz}, and, for every compact set $T\subset\R_+$, there
exist $\gamma_1<1$ or $\gamma_2>1$ such that
\begin{equation}\label{eq:dissipativity_1}
\langle Df(t,\xi)-Df(t,\zeta),\xi-\zeta\rangle\leq\gamma_1\|\xi-\zeta\|^2\quad
\forall\,t\in T,\,\,\forall\,\xi,\zeta\in\R^p
\end{equation}
or
\begin{equation}\label{eq:dissipativity_2}
\langle Df(t,\xi)-Df(t,\zeta),\xi-\zeta\rangle\geq\gamma_2\|\xi-\zeta\|^2\quad
\forall\,t\in T,\,\,\forall\,\xi,\zeta\in\R^p.
\end{equation}
Then statements \ref{ls:single_valued_s_1} and \ref{ls:single_valued_s_2} of \Cref{thm:single_valued} hold
and \eqref{eq:lure_ss} has the uniqueness property.
\end{corollary}
Note that if $Df_t$ is dissipative for every $t\geq 0$, then \eqref{eq:dissipativity_1} holds
with $\gamma_1=0$.

{\it Proof of \Cref{cor:dissipativity}.} Let $T\subset\R_+$ be compact and set $\varepsilon:=1-\gamma_1>0$ or $\varepsilon:=\gamma_2-1>0$ depending on whether \eqref{eq:dissipativity_1} or \eqref{eq:dissipativity_2} holds. Then
\[
\|F(t,\xi)-F(t,\zeta)\|\geq\varepsilon\|\xi-\zeta\|\quad\forall\,t\in T,\,\,\xi,\zeta\in\R^p.
\]
By hypothesis, there exists $b>0$ such that $\|F(t,\xi_0)\|\leq b$ for all $t\in T$, and thus
\[
\|F_t(\xi)\|=\|F(t,\xi)\|\geq\varepsilon\|\xi\|-(b+\varepsilon\|\xi_0\|)\quad\forall\,t\in T,\,\,\xi\in\R^p,
\]
showing that $F_t$ is radially unbounded, uniformly in $t$ for $t\in T$. The claim now follows
from \Cref{thm:single_valued_uniqueness}.
\hfill$\Box$

Under the mild extra assumption that $f$ is bounded on compact sets, an alternative version of \Cref{thm:single_valued_uniqueness} is given in the following corollary.
\begin{corollary}\label{cor:single_valued_uniqueness}
Assume that $f$ is bounded on compact sets and locally Lipschitz in the sense of \eqref{eq:lipschitz}. Furthermore, assume that $F_t$ is radially unbounded, locally uniformly in $t$, and,
for all compact sets $T\subset\R_+$ and $K\subset\R^p$, there exists $b>0$ such that $I-
DM$ is invertible with $\|(I-DM)^{-1}\|\leq b$ for all $M\in\bigcup_{(t,\xi)\in T\times K}
({\rm d}^{\rm c}f_t)(\xi)$.
Then statements \ref{ls:single_valued_s_1} and \ref{ls:single_valued_s_2} of \Cref{thm:single_valued} hold and \eqref{eq:lure_ss} has the uniqueness property.
\end{corollary}
\begin{proof} The claim follows from \Cref{thm:single_valued_uniqueness}, provided it can be proved that the assumptions of the corollary imply \eqref{eq:lower_lipschitz}. To this end, note that, by statement \ref{ls:local_global_s_3} of \Cref{thm:local_global}, $F_t$ is a Lipschitz homeomorphism for every $t\geq 0$. Rademacher's theorem guarantees that $F_t$ and $F_t^{-1}$ are differentiable almost everywhere. By \cite{bd24}, we have $N_{F_t^{-1}}=F_t(N_{F_t})$. Hence, for every $\xi\in\R^p\backslash F_t(N_{F_t})$, the maps $F_t^{-1}$ and $F_t$ are differentiable
at $\xi$ and $F_t^{-1}(\xi)$, respectively, and we have that
\begin{equation}\label{eq:inverse_formula}
({\rm d} F_t^{-1})(\xi)=\big(({\rm d}F_t)(F_t^{-1}(\xi))\big)^{-1}\quad\forall\,\xi\in\R^p\backslash F_t(N_{F_t}).
\end{equation}
Let $T\subset\R_+$ and $K\subset\R^p$ be compact. As $f$ and hence, $F$ are bounded on compact sets,
there exists $\rho>0$ such that $F_t(K)\subset\B(0,\rho)$ for all $t\in T$.
By the locally uniform radial unboundedness of $F_t$, we have that there exists a compact
set $\Gamma$ such that $F_t^{-1}(\B(0,\rho))\subset\Gamma$ for all $t\in T$. By hypothesis
there exists $b>0$ such that $\|P^{-1}\|\leq b$ for all $P\in\bigcup_{(t,\xi)\in T\times\Gamma}({\rm d}^{\rm c}F_t)(\xi)$. Appealing to \eqref{eq:inverse_formula}, we obtain
\[
\|({\rm d} F_t^{-1})(\xi)\|\leq b\quad\forall\,\xi\in\B(0,\rho)\backslash F_t(N_{F_t}),\,\,\forall\,t\in T.
\]
An application of statement \ref{ls:derivative_bounds_s_1} of \Cref{lem:derivative_bounds} with $U=\B(0,\rho)$ shows that
\[
\|F_t^{-1}(\zeta_1)-F_t^{-1}(\zeta_2)\|\leq b\|\zeta_1-\zeta_2\|\quad\forall\,\zeta_1,\zeta_2\in\B(0,\rho),\,\,\forall\,t\in T.
\]
Let $\xi_1,\xi_2\in K$, $\xi_1\not=\xi_2$, $t\in T$ and set $\zeta_1:=F_t(\xi_1)$, $\zeta_2:=F_t(\xi_2)$. Then $\zeta_1,\zeta_2\in\B(0,\rho)$ and it follows from the
above inequality that
\[
\|F_t(\xi_1)-F_t(\xi_2)\|\geq (1/b)\|\xi_1-\xi_2\|,
\]
establishing \eqref{eq:lower_lipschitz} with $\varepsilon=1/b$.
\end{proof}

We now come to the proof of \Cref{thm:single_valued_uniqueness}.

{\it Proof of \Cref{thm:single_valued_uniqueness}.}
Let $T\subset\R_+$ and $K\subset\R^p$ be compact. Without loss of generality we may assume that
$\xi_0\in K$. Then there exists $\lambda\in L^1_{\rm loc}(\R_+,\R_+)$ such that \eqref{eq:lipschitz} holds, and thus
\[
\|f(t,\xi)\|\leq c\lambda(t)+\|f(t,\xi_0)\|=:\varphi(t)\quad\forall\,\xi\in K,\,\,\forall\,t\geq 0
\]
where $c>0$ is a suitable constant. The function $\varphi$ is in $L^1_{\rm loc}(\R_+,\R_+)$,
and thus, \eqref{eq:single_valued} holds. Furthermore, \eqref{eq:lower_lipschitz} trivially implies that $F_t$ is locally injective for every $t\geq 0$. Hence the assumptions of \Cref{thm:single_valued} are satisfied, and consequently, statements \ref{ls:single_valued_s_1} and \ref{ls:single_valued_s_2} of \Cref{thm:single_valued} hold.

We proceed to prove uniqueness. Let $v\in L^\infty_{\rm loc}(\R_+,\R^{m_{\rm e}})$ and let
the sets $T\subset\R_+$ and $\Gamma\subset\R^n$ be compact.
The uniqueness property follows from an application of \cite[Appendix C.3]{s98} or \cite[Chapter III, \S 10, Supp.\,II]{w98} to system \eqref{eq:lure_ss'}, provided that the function
\[
h(t,z)=f\big(t,g(t,Cz+D_{\rm e}v(t))\big)=f\big(t,F_t^{-1}(Cz+D_{\rm e}v(t))\big)
\]
introduced in \eqref{eq:lure_ss'} satisfies a Lipschitz condition on $\Gamma$ in the sense that there exists $\kappa\in L^1(T,\R_+)$ such that
\begin{equation}\label{eq:h_lipschitz}
\|h(t,z_1)-h(t,z_2)\|\leq\kappa(t)\|z_1-z_2\|\quad\forall\,z_1,z_2\in\Gamma,\,\,\,
\mbox{for a.e. $t\in T$}.
\end{equation}
To show this, let $\Gamma'\subset\R^p$ be compact and such that $Cz+D_{\rm e}v(t)\in\Gamma'$ for all $z\in\Gamma$ and almost every $t\in T$. Invoking uniform radial unboundedness of $F_t$ for $t\in T$, statement \ref{ls:assumptions_s_1} of \Cref{lem:assumptions} below guarantees that there exists $\rho>0$ such that $F_t^{-1}(\Gamma')\subset\overline{\B(0,\rho)}$ for all $t\in T$. By hypothesis, there exists $\varepsilon>0$ such that \eqref{eq:lower_lipschitz} holds with $K:=\overline{\B(0,\rho)}$. It follows that
\[
\|\xi_1-\xi_2\|\geq\varepsilon\|F_t^{-1}(\xi_1)-F_t^{-1}(\xi_2)\|\quad\forall\,\xi_1,\xi_2\in\Gamma',\,\,
\forall\,t\in T,
\]
implying that
\[
\|F_t^{-1}(Cz_1+D_{\rm e}v(t))-F_t^{-1}(Cz_2+D_{\rm e}v(t))\|\leq (1/\varepsilon)\|z_1-z_2\|\quad
\forall\,z_1,z_2\in\Gamma,\,\,\,\mbox{for a.e. $t\in T$}.
\]
Combining this with the Lipschitz property of $f$, we conclude that there exists  $\kappa\in L^1(T,\R_+)$ such that \eqref{eq:h_lipschitz} is satisfied.
\hfill$\Box$

In the following corollary of \Cref{thm:single_valued_uniqueness}, we consider a scenario
in which we have existence for all initial conditions, uniqueness and forward completeness.
\begin{corollary}\label{cor:single_valued_uniqueness'}
Assume that $f$ is locally Lipschitz in the sense of \eqref{eq:lipschitz} and, for every
compact set $T\subset\R_+$, there exist $\xi_0\in\R^p$, $b>0$ and $\delta>0$ such that $t\mapsto f(t,\xi_0)$ is bounded on $T$, $\|M\|\leq b$ for all $M\in\bigcup_{(t,\xi)\in T\times
\R^p}({\rm d}^{\rm c}f_t)(\xi)$ and
\begin{equation}\label{eq:det_condition}
|\det(I-DM)|\geq\delta\quad\forall\,M\in\bigcup_{(t,\xi)\in T\times
\R^p}({\rm d}^{\rm c}f_t)(\xi).
\end{equation}
Then \eqref{eq:lure_ss} has the uniqueness property, $\tilde\sB(S^f,t_0)=\sB(S^f,t_0)$ for all $t_0\geq 0$, and, for all $v\in L^\infty_{\rm loc}(\R_+,\R^{m_{\rm e}})$, $t_0\geq 0$ and $x^0\in\R^n$, there exists a unique $(v,x,y)\in \sB(S^f,t_0)$ such that $x(t_0)=x^0$.
\end{corollary}
\begin{proof} Let $T$ be an arbitrary compact subset of $\R_+$. It follows from the hypotheses that $\|({\rm d}f_t)(\xi)\|\leq b$ for all $\xi\in\R^p\backslash N_{f_t}$ and all $t\in T$. By statement \ref{ls:derivative_bounds_s_1}
of \Cref{lem:derivative_bounds}, $f_t$ is globally Lipschitz with Lipschitz constant $b$ for every $t\in T$. Hence, there exists $a>0$ such that
\begin{equation}\label{eq:sv_unique}
\|f(t,\xi)\|\leq \|f(t,\xi_0)\|+b(\|\xi_0\|+\|\xi\|)\leq a+ b(\|\xi_0\|+\|\xi\|)\quad\forall\,t\in T,\,\,\forall\,\xi\in\R^p.
\end{equation}
Furthermore, by Cramer's rule,
\[
(I-DM)^{-1}=\frac{1}{\det(I-DM)}{\rm adj}(I-DM)\quad\forall\,M\in \bigcup_{(t,\xi)\in T\times
\R^p}({\rm d}^{\rm c}f_t)(\xi),
\]
where adj denotes the adjugate.
Combining this with the hypotheses on ${\rm d}^{\rm c}f_t$ shows that there exists $c>0$ such that
\begin{equation}\label{eq:sv_unique'}
\|(I-DM)^{-1}\|\leq c\quad\forall\,M\in \bigcup_{(t,\xi)\in T\times
\R^p}({\rm d}^{\rm c}f_t)(\xi).
\end{equation}
It follows from \Cref{pro:radial_unboundedness_diff} below that $F_t=I-Df_t$ is a Lipschitz homeomorphism
for every $t\geq 0$ and $F_t$ is radially unbounded, locally uniformly in $t$. By \cite{bd24}, we have $N_{F_t^{-1}}=F_t(N_{F_t})$, and so,
\[
({\rm d} F_t^{-1})(\xi)=\big(({\rm d}F_t)(F_t^{-1}(\xi))\big)^{-1}\quad\forall\,t\in\R_+,\,\,\forall\,\xi\in\R^p\backslash F_t(N_{F_t}).
\]
Invoking \eqref{eq:sv_unique'}, we conclude that
\[
\|({\rm d} F_t^{-1})(\xi)\|\leq c\quad\forall\,t\in T,\,\,\forall\,\xi\in\R^p\backslash F_t(N_{F_t}).
\]
Another application of statement \ref{ls:derivative_bounds_s_1} of \Cref{lem:derivative_bounds}
yields that $F_t^{-1}$ is globally Lipschitz with Lipschitz constant $c$ for every $t\in T$. Consequently, there exists $\varepsilon>0$ such that
\[
\|F(t,\xi)-F(t,\zeta)\|=\|F_t(\xi)-F_t(\zeta)\|\geq\varepsilon\|\xi-\zeta\|\quad\forall\,t\in T,\,\,\forall\,\xi,\zeta\in\R^p.
\]
We have now shown that the hypotheses of \Cref{thm:single_valued_uniqueness} are satisfied. Therefore, an application
of \Cref{thm:single_valued_uniqueness} yields that
\eqref{eq:lure_ss} has the uniqueness property, and,
for all $v\in L^\infty_{\rm loc}(\R_+,\R^{m_{\rm e}})$, $t_0\geq 0$ and $x^0\in\R^n$, there exists a unique $(v,x,y)\in \tilde\sB(S^f,t_0)$ such that $x(t_0)=x^0$.

It remains to show that $\tilde\sB(S^f,t_0)=\sB(S^f,t_0)$. Let $v\in L^\infty_{\rm loc}(\R_+,\R^{m_{\rm e}})$ be fixed, but arbitrary. As in the proof of \Cref{thm:single_valued}, set $h(t,z):=f\big(t,F_t^{-1}(Cz+D_{\rm e}v(t))\big)$. As $F_t^{-1}$ is globally Lipschitz
with Lipschitz constant $c$ for every $t\in T$, it follows from \eqref{eq:sv_unique} that there exist $c_1, c_2, c_3>0$ such that
\[
\|h(t,z)\|=\|f(t,F_t^{-1}(Cz+Dv(t))\|\leq c_1+c_2\|F_t^{-1}(0)\|+c_3\|z\|\quad\forall\,t\in T,\,\,\forall\,z\in\R^n.
\]
As $F_t$ is radially unbounded, uniformly in $t$ for $t\in T$, it is clear that the function $t\mapsto F_t^{-1}(0)$ is bounded on $T$. Consequently, for every compact set $T\subset\R_+$, there exists $c_4>0$ such that
\[
\|h(t,z)\|\leq c_4(1+\|z\|)\quad\forall\,t\in T,\,\,\forall\,z\in\R^n.
\]
An application of \cite[Proposition 2.1.19]{hp05} to \eqref{eq:lure_ss'} shows
that $\tilde\sB(S^f,t_0)=\sB(S^f,t_0)$, completing the proof.
\end{proof}
We present some examples of nonlinearities $f$ and feedthrough matrices $D$ such that $F$ satisfies
the hypotheses  of Theorems \ref{thm:single_valued} or \ref{thm:single_valued_uniqueness} or any of their corollaries.
\begin{example}\label{ex:single_valued}
\begin{enumerate}[label = {\bf (\alph*)}, ref={\rm (\alph*)}, leftmargin = 0ex, itemindent = 4.6ex]
\begin{rm}
\item\label{ls:single_valued_ex_a}
Consider the case $p=m=2$, $D=I$ and define $f:\R_+\times\R^2\to\R^2$ by
\begin{equation}\label{eq:ex_a}
f(t,\xi):=\xi-g(\|\xi\|)R(\theta(t))\xi,
\end{equation}
where the functions $g:\R_+\to\R_+$ and $\theta:\R_+\to\R$ are continuous and measurable, respectively, and $R$ is the two-dimensional rotation matrix
\[
R(\omega):=\begin{pmatrix}\cos\omega & -\sin\omega\\
\sin\omega & \cos\omega\end{pmatrix},\qquad\omega\in\R.
\]
We assume that the function $\R_+\to\R_+,\,s \mapsto sg(s)$ is injective and radially unbounded.

It is obvious that $f$ satisfies \eqref{eq:single_valued}. As $R^\top(\theta(t))R(\theta(t))=I$ for all $t\geq 0$, we have that
\[
\|F_t(\xi)\|^2=g^2(\|\xi\|)\langle\xi,R^\top(\theta(t))R(\theta(t))\xi\rangle=g^2(\|\xi\|)\|\xi\|^2,\quad
\forall\,t\geq 0,\,\,\forall\,\xi\in\R^2.
\]
It follows from the hypotheses on $g$ that $F_t$ is injective for each $t\geq 0$
and $F_t$ is radially unbounded, uniformly in $t$. Consequently, statements \ref{ls:single_valued_s_1} and \ref{ls:single_valued_s_2} of \Cref{thm:single_valued} apply to \eqref{eq:lure_ss} when $D=I$ and $f$ is given by \eqref{eq:ex_a}.

Furthermore, for $t\in
\theta^{-1}(\pi/2)$, the matrix $R(\theta(t))$ is skew-symmetric, and thus, $\langle\xi,R(\theta(t))\xi\rangle=0$
for all $\xi\in\R^2$. A straightforward calculation shows that
\[
\|f(t,\xi)\|^2=\|\xi\|^2+g^2(\|\xi\|)\|\xi\|^2,\quad\forall\,t\in\theta^{-1}(\pi/2),\,\,\forall\,
\xi\in\R^2.
\]
Therefore, $\|f(t,\xi)\|^2/\|\xi\|^2=1+g^2(\|\xi\|)\geq 1$, whenever $t\in\theta^{-1}(\pi/2)$.
As $\|D\|=1$, we see that the condition in \Cref{cor:single_valued} is not satisfied, and hence the conclusion of this corollary does not apply.
\item\label{ls:single_valued_ex_b}
Let us consider a variant of the previous example, namely the situation wherein
$p=m$ and $f:\R_+\times\R^p\to\R^p$ is given by
\[
f(t,\xi):=g(\|\xi\|)J(t)\xi,
\]
where $g:\R_+\to\R$ is continuously differentiable and $J:\R_+\to {\rm O}(p,\R)$
is measurable, with ${\rm O}(p,\R)$ denoting the group of orthogonal real matrices of format $p\times p$. It is clear that $f$ satisfies \eqref{eq:lipschitz} as
\[
\|f(t,\xi_1)-f(t,\xi_2)\|\leq\big\| g(\|\xi_1\|) - g(\|\xi_2\|)\big\|\|J(t)\xi_1 \| + g(\|\xi_2\|)\big \|J(t)\xi_1-J(t)\xi_2\big\|\quad\forall\,\xi_1,\xi_2\in\R^p.
\]
Furthermore, a routine calculation shows that
\[ (\rd f_t)(\xi) = \left\{ \begin{aligned}
&g(0)J(t),   &   \xi  &=0\\
&\frac{g'(\|\xi\|)}{\|\xi\|}(J(t)\xi)\xi^\top + g(\|\xi\|)J(t), &  \xi&\neq 0\,.
\end{aligned} \right.
\]
Let us now focus on the specific case wherein $g(s) = 1/\sqrt{1+s^2}$. It is clear that
$F_t$ is radially unbounded, uniformly in $t$. Furthermore, as $g(0) = 1$ and $g'(s) = -s(1+s^2)^{-\tfrac{3}{2}}$, it follows that
\[
(\rd f_t)(\xi) = \frac{1}{\sqrt{1+\|\xi\|^2}}J(t)(I - P(\xi)) \quad \forall\,\xi\in\R^p,\quad
\mbox{where $P(\xi):=(1+\|\xi\|^2)^{-1}\xi\xi^\top$.}
\]
As $P(\xi)$ is a symmetric positive semi-definite matrix with eigenvalues in the interval $[0,1]$,
it follows that $\|I - P(\xi)\|\leq 1$ for all $\xi\in\R^p$, and thus,
\[
\|D(\rd f_t)(\xi)\|\leq\|D\|\quad\forall\,\xi\in\R^p,\,\,\forall\,t\geq 0.
\]
Consequently, if $\|D\|<1$, then there exists $\varepsilon>0$ such that $\|F(t,\xi)-F(t,\zeta)\|\geq\varepsilon\|\xi-\zeta\|$ for all $t\geq 0$ and $\xi,\zeta\in\R^p$,
in particular \eqref{eq:lower_lipschitz} holds. We conclude that if $g(s) = 1/\sqrt{1+s^2}$
and $\|D\|<1$, then \Cref{thm:single_valued_uniqueness} and Corollaries \ref{cor:single_valued_uniqueness} and \ref{cor:single_valued_uniqueness'} are applicable.
\item\label{ls:single_valued_ex_c}
Consider the case $p = m$, $D=I$ with nonlinearity
\begin{equation}\label{eq:single_valued_exa}
f:\R_+\times \R^p\to\R^p,\,\,(t,\xi)\mapsto \frac{h(t)}{1+\|\xi\|}\xi,
\end{equation}
where $h:\R_+\to\R$ is measurable, locally bounded and $h(t)\in(-\infty,1]$ for all $t\geq 0$.
The map $F$ is given by
\[
F(t,\xi) =\xi-Df(t,\xi)=\xi-f(t,\xi) =\frac{1-h(t)+\|\xi\|}{1+\|\xi\|}\xi
\quad\forall\,\xi\in\R^p,\,\,\forall\,t\geq 0.
\]
Obviously, $F_t$ is radially unbounded, locally uniformly in $t$. It is also readily checked that $F_t$ is injective for every $t\geq 0$. As $f$ also satisfies \eqref{eq:single_valued} and is
bounded, locally uniformly in $t$, it follows that \Cref{thm:single_valued} is applicable to \eqref{eq:lure_ss} with $D=I$ and $f$ given by \eqref{eq:single_valued_exa}. We remark that \Cref{cor:single_valued_uniqueness'} is also applicable, provided that $h(t)$ is bounded away from $1$. As $({\rm d} F_t)(0)=(1-h(t))I$ for all $t\geq 0$, it follows that \eqref{eq:det_condition} does not hold if $h(t)=1$, in which case \Cref{cor:single_valued_uniqueness'} is not applicable.
\hfill $\Diamond$
\end{rm}
\end{enumerate}
\end{example}
\subsection{$F_t$ failing to be invertible for all $t$ in a set of positive measure}\label{subsec:not_invertible}
In the following, we present an approach to well-posedness of \eqref{eq:lure_ss} which is based on set-valued analysis and does not require $F_t$ to be a local homeomorphism. Whilst the case of $F_t$ failing to be invertible on a set of zero measure is covered by Remark \ref{rem:a.e.}, here we consider the situation wherein invertibility of $F_t$ may fail for all $t\geq 0$ belonging to a set of positive measure. In essence, the approach involves recasting \eqref{eq:lure_ss} as a differential inclusion combined with an application of suitable selection theorems, including Filippov's selection theorem. The symbol $F_t^{-1}$ will now denote the inverse image of $F_t$, more precisely, $F_t^{-1}$ is the set-valued function mapping each point $\xi\in\R^p$ to its fibre $F_t^{-1}(\xi)$.

We introduce the following assumptions.
\begin{enumerate}[label = {\bfseries (A\arabic*)}, ref = {\rm (A\arabic*)}]
\item \label{ls:a1} For every $t\geq 0$, the set~$F_t^{-1}(\xi)=\{\zeta\in\R^p:F_t(\zeta)=\xi\}$ is nonempty for all $\xi\in{\rm im}\,(C,D_{\rm e})$.
\item \label{ls:a2} $F_t$ is radially unbounded, locally uniformly in $t$.
\item \label{ls:a3} For every~$t\geq 0$ and every~$\xi\in {\rm im}\,(C,D_{\rm e})$, the set~$f(t,F_t^{-1}(\xi))=\{f(t,\zeta):\zeta\in F_t^{-1}(\xi)\}$ is convex.

\end{enumerate}
A discussion of these assumptions and examples of functions satisfying \ref{ls:a1}--\ref{ls:a3} appear after the statement of the following theorem, the main result of this section. The theorem shows in particular that, if \ref{ls:a1}--\ref{ls:a3} hold, then \eqref{eq:lure_ss} is well posed in the sense that, for every $v\in L^\infty_{\rm loc}(\R_+,\R^{m_{\rm e}})$, every $t_0\geq 0$ and every $x^0\in\R^n$, there exists $(v,x,y)\in\tilde\sB(S^f,t_0)$ such that $x(t_0)=x^0$. The proof of the theorem can be found at the end of the section.
\begin{theorem}\label{thm:inclusion_wp}
Assume
that \ref{ls:a1}--\ref{ls:a3} are satisfied and, for all compact sets $K\subset\R^p$, there exists $\varphi\in L^1_{\rm loc}(\R_+,\R_+)$ such that \eqref{eq:single_valued} holds.
Let $v\in L^\infty_{\rm loc}(\R_+,\R^{m_{\rm e}})$, $t_0\geq 0$ and $x^0\in\R^n$. The following statements hold.
\begin{enumerate}[label = {\bf (\arabic*)}, ref={\rm (\arabic*)}, leftmargin = 0ex, itemindent = 4.6ex]
\item \label{ls:inclusion_wp_s_1} There exists $(v,x,y)\in\tilde\sB(S^f,t_0)$ such that $x(t_0)=x^0$.
\item \label{ls:inclusion_wp_s_2}
If $(v,x,y)\in\tilde\sB(S^f,t_0)$ has a bounded maximal interval of definition $[t_0,\tau)$,
where $t_0<\tau<\infty$, then $\int_{t_0}^\tau\|Bf(t,y(t))\|{\rm d}t=\infty$, $y\not\in L^\infty([0,\tau),\R^p)$ and $\|x(t)\|\to\infty$ as $t\uparrow\tau$. In particular,
system \eqref{eq:lure_ss} has the blow-up property.
\end{enumerate}
\end{theorem}
Note that if $f$ satisfies a suitable Lipschitz condition, then it follows from
\Cref{thm:single_valued_uniqueness} that the uniqueness property holds provided
that $F$ satisfies \eqref{eq:lower_lipschitz}.
\begin{remark}\label{rem:a.e.'}
\begin{rm}
Consider the following assumptions which are slightly weaker than \ref{ls:a1}-\ref{ls:a3}.
\begin{enumerate}[label = {\bfseries (B\arabic*)}, ref = {\rm (B\arabic*)}]
\item \label{ls:a1'} For almost every $t\geq 0$, the set~$F_t^{-1}(\xi)=\{\zeta\in\R^p:F_t(\zeta)=\xi\}$ is nonempty for all $\xi\in{\rm im}\,(C,D_{\rm e})$.
\item \label{ls:a2'} $F_t$ is radially unbounded, locally essentially uniformly in $t$, in the sense of \Cref{rem:a.e.}.
\item \label{ls:a3'} For almost every~$t\geq 0$ and every~$\xi\in {\rm im}\,(C,D_{\rm e})$, the set~$f(t,F_t^{-1}(\xi))=\{f(t,\zeta):\zeta\in F_t^{-1}(\xi)\}$ is convex.
\end{enumerate}
We claim that \Cref{thm:inclusion_wp} remains valid if assumptions \ref{ls:a1}-\ref{ls:a3} are replaced by \ref{ls:a1'}-\ref{ls:a3'}. Indeed, if \ref{ls:a1'}-\ref{ls:a3'} hold, then
there exists a null set $E\subset\R_+$ such that the  Caratheodory functions $\hat f$ and $\hat F$ defined in \Cref{rem:a.e.} satisfy \ref{ls:a1}-\ref{ls:a3}. Consequently, \Cref{thm:inclusion_wp} applies to system \eqref{eq:lure_ss} with $f$ replaced by $\hat f$
and the claim follows from the identity  $\tilde\sB(S^{\hat f},t_0)=
\tilde\sB(S^f,t_0)$.
\hfill$\Diamond$
\end{rm}
\end{remark}

The following lemma identifies certain properties of $F_t$ which are related to \ref{ls:a1}-\ref{ls:a3} in the sense that they imply or are implied by some of the
assumptions \ref{ls:a1}-\ref{ls:a3}.
\begin{lemma}\label{lem:assumptions}
The following statements hold.
\begin{enumerate}[label = {\bf (\arabic*)}, ref={\rm (\arabic*)}, leftmargin = 0ex, itemindent = 4.6ex]
\item\label{ls:assumptions_s_1}
Assumption \ref{ls:a2} holds if, and only if, for all compact sets $T\subset\R_+$ and $K\subset\R^p$, there exists $\rho\geq 0$ such that $F_t^{-1}(K)\subset\B(0,\rho)$ for all
$t\in T$.
\item\label{ls:assumptions_s_2}
If \ref{ls:a1} and \ref{ls:a2} are satisfied, then the set $F_t^{-1}(\xi)$ is non-empty and compact for all $t\geq 0$ and all $\xi\in {\rm im}\,(C,D_{\rm e})$ and the set-valued map $\xi\mapsto F_t^{-1}(\xi)$ is upper semicontinuous on ${\rm im}\,(C,D_{\rm e})$.
\item\label{ls:assumptions_s_3}
If, for every $t\geq 0$, the map $F_t:\R^p\to\R^p$ is locally injective and
\ref{ls:a2} is satisfied, then $F_t$ is a homeomorphism and \ref{ls:a1} and \ref{ls:a3} hold.
\item\label{ls:assumptions_s_4} If $f$ does not depend on $t$ and $F:\R^p\to\R^p$ is a homeomorphism, then \ref{ls:a1}-\ref{ls:a3} are satisfied.
\end{enumerate}
\end{lemma}
\begin{proof}
\ref{ls:assumptions_s_1} Assume that \ref{ls:a2} holds.
Let $T\subset\R_+$ and $K\subset\R^p$ be compact and let $\sigma>\sup_{\xi\in K}\|\xi\|$.
There exists $\rho>0$ such that $\|F_t(\xi)\|\geq\sigma$ for all $t\in T$ and all $\xi\in\R^p$
with $\|\xi\|\geq\rho$. Let $t\in T$ and $\xi\in F_t^{-1}(K)$, then $F_t(\xi)\in K$, and so
$\|F_t(\xi)\|<\sigma$, whence $\|\xi\|<\rho$. It follows that $F_t^{-1}(K)\subset\B(0,\rho)$
for all $t\in T$.

To prove the converse, let $T\subset\R_+$ be compact and $\sigma>0$. By hypothesis, there
exists $\rho>0$ such that $F_t^{-1}(\overline{\B(0,\sigma)})\subset\B(0,\rho)$. Consequently,
for $t\in T$ and $\xi\in\R^p$ with $\|\xi\|\geq\rho$, we have that $\xi\not\in F_t^{-1}(\overline{\B(0,\sigma)})$, implying that $\|F_t(\xi)\|>\sigma$. Hence, $F_t$ is
radially unbounded, locally uniformly in $t$, showing that \ref{ls:a2} is satisfied.

\ref{ls:assumptions_s_2} By \ref{ls:a1}, $F_t^{-1}(\xi)$ is non-empty for all $\xi\in
{\rm im}\,(C,D_{\rm e})$ and compactness follows from \ref{ls:a2}, statement \ref{ls:assumptions_s_1} and the continuity
of $F_t$. The upper semicontinuity of $F_t^{-1}$ is a consequence of \Cref{lem:usc}
in the Appendix.

\ref{ls:assumptions_s_3} This is an immediate consequence of statement \ref{ls:local_global_s_1}
of \Cref{thm:local_global}.

\ref{ls:assumptions_s_4} As $F^{-1}$ is a continuous map, $F^{-1}(K)$ is compact for every compact set $K\subset\R^p$, and it follows from statement \ref{ls:assumptions_s_1} that \ref{ls:a2} holds. Assumptions \ref{ls:a1} and \ref{ls:a3} are trivially satisfied.
\end{proof}
In the remark below we provide some further commentary on assumptions \ref{ls:a1}--\ref{ls:a3}.
\begin{remark}\label{rem:assumptions}
\begin{rm}
\begin{enumerate}[label = {\bf (\alph*)}, ref={\rm (\alph*)}, leftmargin = 0ex, itemindent = 4.6ex]
\item\label{ls:rem_assumptions_s_1}
In the definition of trajectories or pre-trajectories $(v,x,y)$, the output $y$ is required
to be locally integrable on its interval of definition. However, if \ref{ls:a2} holds, then, as
$F_t(y(t))=Cx(t)+D_{\rm e}v(t)$ and $Cx+D_{\rm e}v$ is locally essentially bounded, it follows
that the function $y$ is not only locally integrable, but locally essentially bounded on its interval of definition.

\item\label{ls:rem_assumptions_s_2}
By \Cref{lem:usc} in the Appendix, assumption \ref{ls:a2} implies the upper semicontinuity of $F_t^{-1}$ for every $t\geq 0$. Furthermore, if $f$ does not depend on $t$, all fibres of $F$ are bounded and ${\rm im}\,F$ is closed, then upper semicontinuity of $F^{-1}$
implies \ref{ls:a2}.
\item \label{ls:rem_assumptions_s_3}
A sufficient condition for \ref{ls:a2} to hold is
\[
\limsup_{\|\xi\|\to\infty}\frac{\|Df(t,\xi)\|}{\|\xi\|}<1\quad\forall\,t\geq 0.
\]
The above inequality certainly holds if $f$ is of sublinear growth, that is, if $\lim_{\|\xi\|\to\infty}\|f(t,\xi)\|/\|\xi\|=0$ for all $t\geq 0$.
\item \label{ls:rem_assumptions_s_4}
If $F_t^{-1}(\xi)$ is a singleton for all $t\geq 0$ and all $\xi \in {\rm im}\,(C,D_{\rm e})$, then \ref{ls:a3} is trivially satisfied. Furthermore, as
$Df(t,F_t^{-1}(\xi))=F_t^{-1}(\xi)-\xi$, we see that \ref{ls:a3} holds if
$D$ is left invertible and $F_t^{-1}(\xi)$ is convex for all $t\geq 0$ and all $\xi
\in {\rm im}\,(C,D_{\rm e})$.
\hfill$\Diamond$
\end{enumerate}
\end{rm}
\end{remark}
Next we present examples of nonlinearities $f$ for which \ref{ls:a1}-\ref{ls:a3} are satisfied.
\begin{example}
\begin{rm}
{\bf (a)}
Let $m=p=1$, $D=1$ and let $d:\R_+\to[0,1]$ be measurable. Consider the following saturation nonlinearity with a deadzone of time-dependent width $2d$:
\[
f(t,\xi):=\left\{ \begin{aligned} & 0, & & |\xi|\leq d(t),\\
& \xi-({\rm sign}\,\xi)d(t), & & d(t)<|\xi|\leq 1+d(t),\\
& {\rm sign}\,\xi, & & |\xi|>1+d(t).
\end{aligned} \right.
\]
The function $F(t,\xi)=\xi-f(t,\xi)$ is then given by
\[
F(t,\xi)=\left\{ \begin{aligned} & \xi, & & |\xi|\leq d(t),\\
& ({\rm sign}\,\xi)d(t), & & d(t)<|\xi|\leq 1+d(t),\\
& \xi-{\rm sign}\,\xi, & & |\xi|>1+d(t).
\end{aligned} \right.
\]
It is clear that $F_t$ is radially unbounded, uniformly in $t$. Moreover, for all $t\geq 0$ and all $\xi\in\R$ such that $|\xi|\not=d(t)$, the set $F_t^{-1}(\xi)$ is a singleton, $F_t^{-1}(d(t))=[d(t),1+d(t)]$ and $F_t^{-1}(-d(t))=[-(1+d(t)),-d(t)]$. Finally, as $f\big(t,F_t^{-1}(d(t))\big)=[0,1]$ and $f\big(t,F_t^{-1}(-d(t))\big)=[-1,0]$, we conclude
that assumptions \ref{ls:a1}-\ref{ls:a3} are satisfied.

{\bf (b)} Let $m=p$, $D=(1/2)I$ and consider the time-independent nonlinearity
\[
f(\xi):=\left\{ \begin{aligned} & \xi, & & \|\xi\|\leq 1,\\
& (2-1/\|\xi\|)\xi, & & 1<\|\xi\|\leq 2,\\
& (1+1/\|\xi\|)\xi, & & 2<\|\xi\|<\infty.\end{aligned} \right.
\]
Then, for all $\xi\in\R^p$,
\[
F(\xi)=\xi-Df(\xi):=\left\{ \begin{aligned} & \xi/2, & & \|\xi\|<1,\\
& \xi/(2\|\xi\|), & & 1\leq\|\xi\|\leq 2,\\
& (1-1/\|\xi\|)\xi/2, & & 2<\|\xi\|<\infty,\end{aligned} \right.
\]
and so,
\[
F^{-1}(\xi)=\left\{ \begin{aligned} & \{2\xi\}, & & \|\xi\|<1/2,\\
& \{r\xi:r\in[2,4]\}, & & \|\xi\|=1/2,\\
& \{(2+1/\|\xi\|)\xi\}, & & \|\xi\|>1/2.
\end{aligned} \right.
\]
It is obvious that \ref{ls:a1} and \ref{ls:a2} are satisfied. Since
the set
\[
f(F^{-1}(\xi))=\left\{ \begin{aligned} & \{2\xi\}, & & \|\xi\|<1/2,\\
& \{r\xi:r\in[2,6]\}, & & \|\xi\|=1/2,\\
& \{2(1+1/\|\xi\|)\xi\}, & & \|\xi\|>1/2
\end{aligned} \right.
\]
is convex for every $\xi\in\R^p$, we conclude that \ref{ls:a3} also holds.

{\bf (c)} Let $m=p$, $D=I$, let $h:\R_+\to[0,1]$ be measurable and define a time-dependent
saturation nonlinearity $f:\R_+\times\R^p\to\R^p$ by
\begin{equation}\label{eq:saturation}
f(t,\xi):=h(t)\left\{ \begin{aligned} & \xi, & & \|\xi\|\leq 1,\\
& \xi/\|\xi\|, & & \|\xi\|>1.\end{aligned} \right.
\end{equation}
By routine calculations, we obtain for the pre-images $F_t^{-1}(\xi)$:
\[
F_t^{-1}(\xi)=\{\xi\} \quad \text{if} \quad t\in h^{-1}(0)\,,
\]
\[
F_t^{-1}(\xi)=\left\{ \begin{aligned} & \{[(1/(1-h(t))]\xi\}, & & \|\xi\|\leq 1-h(t),\\
& \{(1+h(t)/\|\xi\|)\xi\}, & & \|\xi\|>1-h(t),\end{aligned} \right.
\qquad\mbox{if $t\in h^{-1}\big((0,1)\big)$},
\]
and
\[
F_t^{-1}(\xi)=\left\{ \begin{aligned} & \overline{\B(0,1)}, & & \xi=0,\\
& \{(1+1/\|\xi\|)\xi\}, & & \xi\not=0,\end{aligned} \right.
\qquad\mbox{if $t\in h^{-1}(1)$}.
\]
It is clear that \ref{ls:a1} and \ref{ls:a2} are satisfied. As $F_t^{-1}(\xi)$ is convex for all $t\geq 0$ and all $\xi\in\R^p$ and $f(t,F_t^{-1}(\xi))=F_t^{-1}(\xi)-\xi$, we see that the set $f(t,F_t^{-1}(\xi))$
is convex for all $t\geq 0$ and all $\xi\in\R^p$, showing that \ref{ls:a3} also holds.
\hfill$\Diamond$
\end{rm}
\end{example}
The following lemma will be used in the proof of \Cref{thm:inclusion_wp}.
\begin{lemma}\label{lem:measurability}
For every measurable function $w:\R_+\to\R^p$, the set-valued function
$\R_+\rightrightarrows\R^p,\,\,t\mapsto F_t^{-1}(w(t))$ is measurable.
\end{lemma}
\begin{proof}
Let $w:\R_+\to\R^p$ be a measurable function. Obviously,
\[
F_t^{-1}(w(t))=\{\xi\in\R^p:F_t(\xi)=w(t)\}=\{\xi\in\R^p:F(t,\xi)=w(t)\}.
\]
As $F$ is a Caratheodory function, an application of \cite[Theorem 8.2.9]{af09} yields
the claim.
\end{proof}
We are now in the position to prove \Cref{thm:inclusion_wp}.
\begin{proof}[Proof of~\Cref{thm:inclusion_wp}]
Let $t_0\geq 0$, $x^0\in\R^n$ and~$v\in L^\infty_{\rm loc}(\R_+,\R^p)$.

\ref{ls:inclusion_wp_s_1} Define the set-valued map
\[ \Phi:[t_0,\infty)\times\R^n \rightrightarrows  \mR^n, \quad (t,z)\mapsto Az+Bf\big(t,F_t^{-1}(Cz+D_{\rm e}v(t))\big)+B_{\rm e}v(t)\,.\]
We claim that
\begin{enumerate}[label = (\roman*)]
    \item $\Phi$ has nonempty, compact and convex values and, for each $t\geq t_0$, $z\mapsto
     \Phi(t,z)$ is upper semicontinuous on ${\rm im}\,(C,D_{\rm e})$;
    \item  for each $z\in\R^n$, the function $t\mapsto\Phi(t,z)$ has a measurable selection;
    \item for every compact set $X\subset\R^n$, there exists $\psi\in L^1_{\rm loc}(\R_+,\R_+)$ such that
        \[
        \sup\{\|\zeta\|:\zeta\in\Phi(t,z),\,z\in X\}\leq\psi(t)\quad\mbox{for a.e. $t\geq 0$.}
        \]
\end{enumerate}
To prove claim (i), note that, by statement \ref{ls:assumptions_s_2} of \Cref{lem:assumptions}, the map $F_t^{-1}$ is upper semicontinuous on ${\rm im}\,(C,D_{\rm e})$
and $F_t^{-1}(\xi)$ is nonempty and compact for all $t\geq 0$ and all $\xi\in {\rm im}\,(C,D_{\rm e})$. Hence, the continuity of $f(t,\,\cdot\,)$ implies that, for each $t\geq 0$, the set-valued map~$\xi\mapsto f(t, F_t^{-1}(\xi))$ is upper semicontinuous on ${\rm im}\,(C,D_{\rm e})$ with nonempty compact values. By \ref{ls:a3}, the set~$f(t,F_t^{-1}(\xi))$ is also convex for every~$t\geq 0$ and all~$\xi\in {\rm im}\,(C,D_{\rm e})$. Combined, these properties yield claim (i).

As for claim (ii), we invoke \Cref{lem:measurability}, to obtain that the map $t\mapsto F_t^{-1}(Cz+D_{\rm e}v(t))$ is measurable. Since it has nonempty and closed values, it follows that it has a measurable selection $w$ (see, for example, \cite[Proposition 3.2]{d92}). The Caratheodory property of $f$ guarantees that the function $t\mapsto f(t,w(t))$ is measurable, and hence is a measurable selection of $t\mapsto f(t,F_t^{-1}(Cz+D_{\rm e}v(t)))$. Consequently, for each $z\in\R^n$, the map $t\mapsto\Phi(t,z)$ has a measurable selection, showing that (ii) holds.

To establish claim (iii), let $\tau>0$ and $X\subset\R^n$ be compact. It is clear that there exists a compact set $K\subset\R^p$ and a null set $E\subset [0,\tau]$ such that
\[
Cz+D_{\rm e}v(t)\in K\quad\forall\,z\in X,\,\,\forall\,t\in [0,\tau]\backslash E.
\]
By assumption \ref{ls:a2} and statement \ref{ls:assumptions_s_1} of \Cref{lem:assumptions} there exists $c>0$ such that
\[
F_t^{-1}(K)\subset\B(0,c)\quad\forall\,t\in [0,\tau].
\]
Invoking \eqref{eq:single_valued}, we have that
\[
\sup\{\|f(t,\xi)\|:\xi\in\overline{\B(0,c)}\}\leq\varphi(t)\quad\forall\,t\geq 0
\]
for suitable $\varphi\in L^1_{\rm loc}(\R_+,\R_+)$. Therefore,
\[
\sup\{\|f(t,\xi)\|:\xi\in F_t^{-1}(Cz+D_{\rm e}v(t)),\, z\in X\}
\leq \sup\{\|f(t,\xi)\|:\xi\in F_t^{-1}(K)\}\leq\varphi(t)
\quad\forall\,t\in[0,\tau]\backslash E.
\]
Since $\tau>0$ was arbitrary, we may conclude that there exists $\psi\in L^1_{\rm loc}(\R_+,\R_+)$
such that (iii) holds.

It follows from claims (i)-(iii) and the theory of differential inclusions that the initial-value problem
\begin{equation}\label{eq:lure_inclusion}
\dot x(t)\in Ax(t)+Bf\big((t,F_t^{-1}(Cx(t)+D_{\rm e}v(t))\big)+B_{\rm e}v(t)=\Phi(t,x(t)),\quad x(t_0)=x^0
\end{equation}
has an absolutely continuous solution $x$ defined on a maximal interval of existence $[t_0,\tau)$,
where $t_0<\tau\leq\infty$, and
\begin{equation}\label{eq:blow_up}
\tau<\infty\quad\Longrightarrow\quad\limsup_{t\uparrow \tau}\|x(t)\|=\infty,
\end{equation}
see, for example, \cite[Corollary 5.2]{d92}. Setting $\tilde f:=Bf$ and $Y(t)=F_t^{-1}(Cx(t)+D_{\rm e}v(t))$ for $t\in[0,\tau)$, we have that
\begin{equation}\label{eq:lure_inclusion'}
\dot x(t)-Ax(t)-B_{\rm e}v(t)\in\tilde f(t,Y(t))=\{\tilde f(t,\xi):\xi\in Y(t)\},\quad\mbox{for a.e. $t\in[0,\tau)$}.
\end{equation}
As observed before, in the proof of claim (ii), the set-valued function $Y$ is measurable. Furthermore, the values of $Y$ are compact. It follows from~\eqref{eq:lure_inclusion'} that $\dot x-Ax-B_{\rm e}v$ is a measurable selection of $t\mapsto\tilde f(t,Y(t))$. By  Filippov's selection theorem (see \cite[Theorem 2.3.13]{v00}), there exists a measurable selection $y$ of $Y$ defined on $[t_0,\tau)$ such that
\[
\dot x(t)-Ax(t)-B_{\rm e}v(t)=\tilde f(t,y(t))=Bf(t,y(t))\quad\mbox{for a.e. $t\in[t_0,\tau)$}.
\]
Invoking \ref{ls:a2} together with statement \ref{ls:assumptions_s_1} of \Cref{lem:assumptions}, we see that there exists $\gamma\in L^\infty_{\rm loc}([t_0,\tau),\R_+)$ such that $\sup\{\|\xi\|:\xi\in Y(t)\}\leq\gamma(t)$ for almost every $t\in[t_0,\tau)$,
implying that $y\in L^\infty_{\rm loc}([t_0,\tau),\R^p)$. Combining this with \eqref{eq:single_valued}
shows that $f\circ y\in L^1_{\rm loc}([t_0,\tau),\R^p)$. Furthermore,
\begin{equation}\label{eq:lure_inclusion_p1}
     y(t)-Df(t,y(t))=F(t,y(t))=F_t(y(t))=Cx(t)+D_{\rm e}v(t)\quad\mbox{for a.e. $t\in[0,\tau)$}.
\end{equation}
Consequently, we have that the triple $(v,x,y)$ is a pre-trajectory of \eqref{eq:lure_ss} satisfying
$x(t_0)=x^0$. Finally, it follows from \eqref{eq:blow_up} that  $(v,x,y)$ is maximally defined and so
$(v,x,y)\in\tilde\sB(S^f,t_0)$, completing the proof of statement~\ref{ls:inclusion_wp_s_1}.

\ref{ls:inclusion_wp_s_2} It is an immediate consequence of statement \ref{ls:inclusion_wp_s_1} and \Cref{pro:blow_up}, that \eqref{eq:lure_ss} has the blow-up property and
\[
\int_{t_0}^\tau\big(\|y(t)\|+\|Bf(t,y(t))\|\big){\rm d}t=\infty.
\]
Furthermore, if $y$ was in $L^\infty([0,\tau),\R^p)$, then, by \eqref{eq:single_valued},
$f\circ y\in L^1([0,\tau),\R^m)$, leading to a contradiction with the divergence of the
above integral.

As $y\not\in L^\infty([0,\tau],\R^p)$, there exists, for every $k\in\N$,
a set $T_k\subset [0,\tau)$ of positive measure such that $\|y(t)\|\geq k$ for all $t\in T_k$. It follows from \ref{ls:a2} that,  for every $l\in\N$, there exists $k_l\in\N$ such that
\[
\|F(t,y(t))\|=\|F_t(y(t))\|\geq l\quad\mbox{for a.e. $t\in T_{k_l}$}.
\]
Using that $Cx(t)=F(t,y(t))-D_{\rm e}v(t)$ for almost every $t\in[t_0,\tau)$,
we conclude that $x$ is unbounded on $[t_0,\tau)$, hence $\limsup_{t\uparrow\tau}\|x(t)\|=\infty$.
Seeking a contradiction, suppose that $\|x(t)\|$ does not converge to $\infty$ as $t\uparrow\tau$.
Then there exist numbers $\rho>0$ and $t_j\in[t_0,\tau)$, $j\in\N$, such that $t_j\to\tau$ as $j\to\infty$ and $x(t_j)\in\B(0,\rho)$ for all $j\in\N$. Let $\varepsilon>0$ and note that, by the
unboundedness of $x$,
\[
S_j:=\{s\in[t_j,\tau):x(s)\not\in\B(0,\rho+\varepsilon)\}\not=\emptyset\quad\forall\,j\in\N.
\]
Setting $s_j:=\inf S_j$, it is obvious that $t_j<s_j<\tau$, $s_j\to\tau$ as $j\to\infty$,
$\|x(s_j)\|=\rho+\varepsilon$ and $x(t)\in\overline{\B(0,\rho+\varepsilon)}$ for all $t\in[t_j,s_j]$ and all $j\in\N$. It is clear that there exists a compact set $\Gamma\subset\R^p$ such that
\[
Cx(t)+D_{\rm e}v(t)\in\Gamma\quad\mbox{for a.e. $t\in[t_j,s_j]$, $j\in\N$}.
\]
Invoking assumption \ref{ls:a2} and \Cref{lem:assumptions}, we see that there exists
a compact set $K\subset\R^p$ such that
\[
F_t^{-1}(\Gamma)\subset K\quad\forall\,t\in[t_0,\tau].
\]
As $y(t)\in F_t^{-1}(Cx(t)+D_{\rm e}v(t))$ for almost every $t\in[t_0,\tau)$, we conclude that
\[
y(t)\in K\quad\mbox{for a.e. $t\in[t_j,s_j]$, $j\in\N$}.
\]
It now follows from \eqref{eq:single_valued} that there exists $\kappa\in L^1([t_0,\tau],\R_+)$ such that
\[
\|f(t,y(t))\|\leq\kappa(t)\quad\mbox{for a.e. $t\in[t_j,s_j]$, $j\in\N$}.
\]
Consequently, routine estimates of $x$ yield
\begin{align*}
\varepsilon\leq\|x(s_j)-x(t_j)\|&\leq \int_{t_j}^{s_j} \| \dot x(s) \| \, \rd s \\
& \leq \big(\|A\|(\rho+\varepsilon)+\|B_{\rm e}\|\|v\|_{L^\infty(t_0,\tau)}\big)
(s_j-t_j)+\|B\|\int_{t_j}^{s_j}\kappa(t){\rm d}t\quad\forall\,j\in\N.
\end{align*}
As the right-hand side converges to $0$ as $j\to\infty$, we obtain a contradiction to the positivity of
$\varepsilon$. Consequently, $\|x(t)\|\to\infty$ as $t\uparrow\tau$. This in turn implies, via the
variation-of-parameters formula for $x$, that $\int_{t_0}^\tau\|Bf(t,y(t))\|{\rm d}t=\infty$.
\end{proof}
Next we provide a condition guaranteeing forward completeness and existence of trajectories for every initial condition.
\begin{corollary}\label{cor:inclusion_wp}
Assume that \ref{ls:a1} and \ref{ls:a3} hold, for all compact sets $K\subset\R^p$, there exists $\varphi\in L^1_{\rm loc}(\R_+,\R_+)$ such that \eqref{eq:single_valued} holds, and,
for all compact sets $T\subset\R_+$, there exist $\rho>0$ and $c>0$ such that $c\|D\|<1$ and
\eqref{eq:al_bound} is satisfied. Then $\tilde\sB(S^f,t_0)=\sB(S^f,t_0)$ for all $t_0\geq 0$ {\rm (}that is, \eqref{eq:lure_ss} is forward complete{\rm )}. Furthermore, for all $v\in L^\infty_{\rm loc}(\R_+,\R^{m_{\rm e}})$, $t_0\geq 0$ and $x^0\in\R^n$, there exists $(v,x,y)\in \sB(S^f,t_0)$ such that $x(t_0)=x^0$.
\end{corollary}
\begin{proof}
As $c\|D\|<1$, it follows from \eqref{eq:al_bound} that $F_t$ is radially unbounded, locally essentially uniformly, in the sense of \Cref{rem:a.e.}, that is, \ref{ls:a2'} in \Cref{rem:a.e.'}
is satisfied. Hence, by  \Cref{rem:a.e.'}, the conclusions of \Cref{thm:inclusion_wp} hold. Therefore, we only need to establish that  $\tilde\sB(S^f,t_0)=\sB(S^f,t_0)$ for all $t_0\geq 0$. To this end, let $(v,x,y)$ be a pre-trajectory defined on $[t_0,\tau)$, where $t_0<\tau<\infty$. It is sufficient to prove that $y\in L^\infty([t_0,\tau),\R^p)$. Indeed, in this case,
statement \ref{ls:inclusion_wp_s_2} of \Cref{thm:inclusion_wp} implies that every maximally defined trajectory starting at $t_0$ is defined on $[t_0,\infty)$ and is therefore in $\sB(S^f,t_0)$.

To show that $y\in L^\infty([t_0,\tau),\R^p)$, note that, by \eqref{eq:al_bound}, there exist $\rho>0$, $c>0$ and a null set $E\subset[t_0,\tau]$
such that $c\|D\|<1$ and
\begin{equation}\label{eq:al_bound'}
\|f(t,\xi)\|\leq c\|\xi\|\quad\forall\,t\in [t_0,\tau]\backslash E,\,\,\forall\,\xi\in\R^p\backslash\B(0,\rho).
\end{equation}
For $t\in [t_0,\tau]\backslash E$, if $\|y(t)\|\geq\rho$, we have
\[
\|y(t)\|\leq\frac{1}{1-c\|D\|}\|Cx(t)+D_{\rm e}v(t)\|,
\]
and thus,
\[
\|y(t)\|\leq\rho+\frac{1}{1-c\|D\|}\|Cx(t)+D_{\rm e}v(t)\|\quad\mbox{for a.e. $t\in[t_0,\tau)$.}
\]
An application of the variation-of-parameters formula for $x$ yields that there
exist constants $c_1,c_2>0$ and a continuous non-negative function $h$ defined
on $[t_0,\tau]$ such that
\begin{equation}\label{eq:y_inequality}
\|y(t)\|\leq c_1+c_2\int_{t_0}^th(s)\|f(s,y(s))\|{\rm d}s\quad\mbox{for a.e. $t\in[t_0,\tau)$.}
\end{equation}
By \eqref{eq:single_valued}, there exists $\varphi\in L^1_{\rm loc}(\R_+,\R_+)$ such that
\[
\sup_{\xi\in\overline{\B(0,\rho)}}\|f(t,\xi)\|\leq\varphi(t)\quad\mbox{for a.e. $t\geq 0$}.
\]
Combining this with \eqref{eq:al_bound'} shows that
\[
\|f(t,y(t))\|\leq\varphi(t)+c\|y(t)\|\quad\mbox{for a.e. $t\in[t_0,\tau)$},
\]
which in conjunction with \eqref{eq:y_inequality} leads to
\[
\|y(t)\|\leq c_3+c_4\int_{t_0}^th(s)\|y(s)\|{\rm d}s\quad\mbox{for a.e. $t\in[t_0,\tau)$,}
\]
where $c_3$ and $c_4$ are suitable positive constants. An application of the Gronwall lemma (see \Cref{lem_a:gronwall} in the Appendix for a suitably general version) yields that $\|y(t)\|\leq c_3+\exp\big(c_4\int_{t_0}^t h(s){\rm d}s\big)$ for almost every $t\in [t_0,\tau)$, whence
$y\in L^\infty([t_0,\tau),\R^p)$.
\end{proof}
\section{Sufficient conditions for the radial unboundedness property}\label{sec:radial_unboundedness}
Radial unboundedness of $F_t$, locally uniformly in $t$, plays an important role
in Theorems \ref{thm:single_valued} and \ref{thm:inclusion_wp}. In this section,
we provide a number of sufficient conditions for this property.
\begin{proposition}\label{pro:radial_unboundedness_diff}
Assume that $f$ is locally Lipschitz in the sense of \eqref{eq:lipschitz} and, for every $M\in
\bigcup_{(t,\xi)\in\R_+\times\R^p}({\rm d^{\rm c}}f_t)(\xi)$, the matrix $I-DM$ is invertible. Assume further that,
for every compact set $T\subset\R_+$, there exist $\xi_0\in\R^p$ and $b>0$ such that
\begin{equation}\label{eq:radial_unboundedness_diff}
\sup_{t\in T}\|Df(t,\xi_0)\|\leq b\quad\mbox{and}\quad \|(I-DM)^{-1}\|\leq b\quad
\forall\,M\in\bigcup_{(t,\xi)\in T\times\R^p}({\rm d}^{\rm c}f_t)(\xi).
\end{equation}
Under these conditions, $F_t$ is a Lipschitz homeomorphism for every $t\geq 0$ and $F_t$ is radially unbounded, locally uniformly in $t$.
\end{proposition}
\begin{proof}
By the Hadamard theorem for locally Lipschitz functions \cite{p82},\footnote[2]{The ``classical'' version of the Hadmard theorem for continuously differentiable functions can be found in many places, see, for example,  \cite[Theorem 15.4]{d85}, \cite{mgz94} or \cite{s80,s02}.}
it follows from the hypotheses that $F_t=I-Df_t$ is a Lipschitz homeomorphism for every $t\geq 0$.
In particular, for each $t\in\R_+$, the map $F_t$ is radially unbounded.
Let $T\subset\R_+$ be compact. We show that the radial unboundedness is uniform in $t$ for $t\in T$.
To this end, we note that Rademacher's theorem guarantees that $F_t$ and $F_t^{-1}$ are differentiable almost everywhere, and, by \cite{bd24}, $N_{F_t^{-1}}=F_t(N_{F_t})$. Hence, for every $\xi\in\R^p\backslash F_t(N_{F_t})$, the maps $F_t^{-1}$ and $F_t$ are differentiable
at $\xi$ and $F_t^{-1}(\xi)$, respectively, and $({\rm d} F_t^{-1})(\xi)=\big(({\rm d} F_t)(F_t^{-1}(\xi))\big)^{-1}$ for all $t\in T$ and all $\xi\in\R^p\backslash F_t(N_{F_t})$.
As $({\rm d} F_t)(F_t^{-1}(\xi))\in ({\rm d}^{\rm c} F_t)(F_t^{-1}(\xi))=I-D({\rm d}^{\rm c} f_t)(F_t^{-1}(\xi))$ for all $\xi\in\R^p\backslash F_t(N_{F_t})$, it follows that
\[
\|({\rm d} F_t^{-1})(\xi)\|\leq b\quad\forall\,t\in T,\,\,\forall\,\xi\in\R^p\backslash F_t(N_{F_t}).
\]
Invoking statement \ref{ls:derivative_bounds_s_1} of \Cref{lem:derivative_bounds}, we conclude that
\[
\|F_t^{-1}(\xi)-F_t^{-1}(\zeta)\|\leq b\|\xi-\zeta\|\quad\forall\,\xi,\zeta\in\R^p,\,\forall\,t\in T,
\]
that is, $F_t^{-1}$ is globally Lipschitz, uniformly in $t$ for $t\in T$. This in turn implies that
\[
\|F_t(\xi)-F_t(\zeta)\|\geq (1/b)\|\xi-\zeta\|\quad\forall\,\xi,\zeta\in\R^p,\,\forall\,t\in T,
\]
and thus, using the first inequality in \eqref{eq:radial_unboundedness_diff},
\[
\|F_t(\xi)\|\geq (1/b)\|\xi-\xi_0\|-b-\|\xi_0\|\geq(1/b)\|\xi\|-\big((1+1/b)\|\xi_0\|+b\big)\quad
\forall\,\xi\in\R^p,\,\forall\,t\in T.
\]
The last inequality shows that $F_t$ is radially unbounded, uniformly in $t$ for $t\in T$.
\end{proof}
\begin{proposition}\label{pro:radial_unboundedness_diff'}
Assume that $f$ is locally Lipschitz in the sense of \eqref{eq:lipschitz}, and, for
all compact sets $T\subset\R_+$ and $K\subset\R^p$, there exist a constant $b>0$,
such that $\|Df(t,\xi)\|\leq b$ for all $t\in T$ and all $\xi\in K$.
If, for all compact $T\subset\R_+$, there exist $\rho>0$ and $b_1<1$ or $b_2>1$
such that
\begin{equation}\label{eq:radial_unboundedness_diff'1}
\langle D({\rm d} f_t)(\xi)\zeta,\zeta\rangle\leq b_1\|\zeta\|^2\quad\forall\,
\zeta\in\R^p,\,\,\forall\,t\in T,\,\,\forall\,\xi\in\big(\R^p\backslash\B(0,\rho)\big)\backslash N_{f_t},
\end{equation}
or
\begin{equation}\label{eq:radial_unboundedness_diff'2}
\langle D({\rm d} f_t)(\xi)\zeta,\zeta\rangle\geq b_2\|\zeta\|^2\quad\forall\,
\zeta\in\R^p,\,\,\forall\,t\in T,\,\,\forall\,\xi\in\big(\R^p\backslash\B(0,\rho)\big)\backslash N_{f_t},
\end{equation}
then $F_t$ is radially unbounded, locally uniformly in $t$.
\end{proposition}
\begin{proof} Let $T\subset\R_+$ be compact and let $\xi\in\R^p\backslash\B(0,\rho)$. Define $\xi_\rho:=(\rho/\|\xi\|)\xi
\in\partial\B(0,\rho)$ and note that $[\xi_\rho,\xi]\subset\R^p\backslash\B(0,\rho)$.
By statement \ref{ls:derivative_bounds_s_2} of \Cref{lem:derivative_bounds},
\[
\langle Df_t(\xi)-Df_t(\xi_\rho),\xi-\xi_\rho\rangle\leq b_1\|\xi-\xi_\rho\|^2\quad
\forall\,\xi\in\R^p\backslash\B(0,\rho),\,\,\forall\,t\in T.
\]
or
\[
\langle Df_t(\xi_\rho)-Df_t(\xi),\xi-\xi_\rho\rangle\leq -b_2\|\xi-\xi_\rho\|^2\quad
\forall\,\xi\in\R^p\backslash\B(0,\rho),\,\,\forall\,t\in T,
\]
depending on whether \eqref{eq:radial_unboundedness_diff'1} or \eqref{eq:radial_unboundedness_diff'2} is satisfied. Setting $\varepsilon:=
1-b_1$ if \eqref{eq:radial_unboundedness_diff'1} holds and  $\varepsilon:=b_2-1$  if \eqref{eq:radial_unboundedness_diff'2} holds, we have that $\varepsilon>0$ and it follows that
\[
\langle F_t(\xi)-F_t(\xi_\rho),\xi-\xi_\rho\rangle\geq\varepsilon\|\xi-\xi_\rho\|^2\quad
\forall\,\xi\in\R^p\backslash\B(0,\rho),\,\,\forall\,t\in T.
\]
By hypothesis, there exists $b>0$ such that $\|Df(t,\zeta)\|\leq b$ for all $\zeta\in \partial\B(0,\rho)$ and $t\in T$, whence $\|F_t(\zeta)\|\leq b+\rho$ for all $\zeta\in\partial\B(0,\rho)$ and $t\in T$. Therefore, for all $\xi\in\R^p\backslash\B(0,\rho)$ and all $t\in T$,
\[
\|F_t(\xi)\|\geq\varepsilon\|\xi-\xi_\rho\|-\|F_t(\xi_\rho)\|
\geq\varepsilon\|\xi-\xi_\rho\|-(b+\rho)
\geq\varepsilon\|\xi\|-\big(b+(\varepsilon+1)\rho\big),
\]
proving the claim.
\end{proof}
%
%
%
%
\appendix
\section{Appendix}
The Appendix contains a proof of \Cref{pro:max_def} and the statement and proofs of two auxiliary results which have been used in the main text.
\subsection{Proof of \Cref{pro:max_def}}
For a given pre-trajectory $(v,x,y)$ on $[t_0,\tau)$, where $0\leq t_0<\tau<\infty$, let
$\sE$ be the set of all~$(X,Y)\in W^{1,1}_{\rm loc}([t_0,T),\R^n)\times L^1_{\rm loc}([t_0,T),\R^p)$ such that $\tau\leq T\leq\infty$,~$(X,Y)|_{[t_0,\tau)}=(x,y)$ and $(v,X,Y)$ is a pre-trajectory of~\eqref{eq:lure_ss} on $[t_0,T)$, where
$T$ depends on~$(X,Y)$. As~$(x,y)\in\sE$, the set~$\sE$ is non-empty. Let~$(x_j,y_j)\in\sE$ be defined on~$[t_0,\tau_j)$, where $\tau\leq\tau_j\leq\infty$, $j=1,2$. We define a partial order on~$\sE$ as follows:
\[
(x_1,y_1)\preceq(x_2,y_2)\quad\Longleftrightarrow\quad\tau_1\leq\tau_2\,\,\,\mbox{and}\,\,\,
(x_2,y_2)|_{[t_0,\tau_1)}=(x_1,y_1).
\]
To prove the claim, we have to show that~$\sE$ has a maximal element, that is, an element~$(x_{\rm m},y_{\rm m})\in\sE$ such that, if~$(X,Y)\in\sE$ and~$(x_{\rm m},y_{\rm m})\preceq
(X,Y)$, then~$(X,Y)=(x_{\rm m},y_{\rm m})$. This we do by an application of Zorn's lemma, by which it is sufficient to show that every totally ordered subset~$\sT\subset\sE$ has an upper bound in~$\sE$, that is, an element~$(x_{\rm u},y_{\rm u})\in\sE$ such that~$(X,Y)\preceq
(x_{\rm u},y_{\rm u})$ for all~$(X,Y)\in\sT$. To this end, let $\sT$ be a totally ordered subset of $\sE$, and, for $(X,Y)\in\sT$, let $[t_0,T_{X,Y})$ be the interval on which $(X,Y)$ is defined, where $\tau\leq T_{X,Y}\leq\infty$. We set
\[
\tau_{\rm u}:=\sup\{T_{X,Y}:(X,Y)\in\sT\}\geq\tau>t_0
\]
and define~$(x_{\rm u},y_{\rm u})$ on~$[t_0,\tau_u)$ by
\[
(x_{\rm u},y_{\rm u})(t):=(X,Y)(t)\quad\forall\:t\in[t_0,T_{X,Y}).
\]
As~$\sT$ is totally ordered, the pair~$(x_{\rm u},y_{\rm u})$ is well-defined, is an element
of~$\sE$ and an upper bound for~$\sT$, completing the proof.
\hfill$\Box$
\subsection{Set-valued maps}
Recall that a map $g:\R^n\to\R^m$ is said to be closed if the image of every closed set is closed.
\begin{lemma}\label{lem:usc}
Let $g:\R^n\to\R^m$ be a function. The following statements hold.
\begin{enumerate}[label = {\bf (\arabic*)}, ref={\rm (\arabic*)}, leftmargin = 0ex, itemindent = 4.6ex]
\item\label{ls:usc_s_1}
If $g$ is closed, then the set-valued map $g^{-1}:\R^m\rightrightarrows\R^n,\,z\mapsto g^{-1}(z)$ {\rm (}that is, each $z$ is mapped to its fibre under $g${\rm )} is upper semicontinuous.
\item\label{ls:usc_s_2}
If $g$ is continuous and radially unbounded, then $g$ is closed.
\end{enumerate}
\end{lemma}
\begin{proof}
\ref{ls:usc_s_1}
Set $G:=g^{-1}$ and let $Y\subset\R^n$ be closed. As $g$ is closed, it is a routine exercise to show
that $\Delta(G)=\{\xi\in\R^m:g^{-1}(\xi)\not=\emptyset\}$ is closed. Consequently, to establish upper
semicontinuity of $G$, it is sufficient to show that
\[
G^{-1}(Y):=\{\xi\in\R^m:G(\xi)\cap Y\not=\emptyset\}=\{\xi\in\Delta(G):G(\xi)\cap Y\not=\emptyset\}
\]
is closed. Let $(\xi_k)_{k\in\N}$ be a convergent sequence in $G^{-1}(Y)$ with limit $\xi\in\R^m$. We have to show that $\xi\in G^{-1}(Y)$. To this end, let $z_k\in G(\xi_k)\cap Y\not=\emptyset$ for all $k\in\N$. As $g(z_k)=\xi_k$ for all $k\in\N$, it follows that $\xi_k\in g(Y)$. By hypothesis, $g(Y)$ is closed, and thus, $\xi\in g(Y)$. Consequently, $\xi=g(y)$ for some $y\in Y$, implying that $y\in g^{-1}(\xi)=G(\xi)$ and hence, $y\in G(\xi)\cap Y$. We conclude that $G(\xi)\cap Y\not=\emptyset$, showing that $\xi\in G^{-1}(Y)$.

\ref{ls:usc_s_2} Let $X\subset\R^n$ be closed and let $(\xi_k)_{k\in\N}$ be a convergent sequence
in $g(X)$ with limit $\xi$. Let $x_k\in X$ be such that $g(x_k)=\xi_k$ for all $k\in\N$. As
$(\xi_k)$ is bounded, so is $(x_k)$, by the radial unboundedness of $g$. Therefore there exists
a convergent subsequence $(x_{k_j})_{j\in\N}$ of $(x_k)$ with limit $x$, where $x\in X$, by the
closedness of $X$. The function $g$ is continuous, and thus,
\[
\xi=\lim_{j\to\infty} \xi_{k_j}=\lim_{j\to\infty} g(x_{k_j})=g(x).
\]
Hence, $\xi\in g(X)$, showing that $g(X)$ is closed.
\end{proof}
\subsection{General version of Gronwall's lemma}
As the version of Gronwall's lemma which has been used in Section \ref{sec:wp} is somewhat more general than what were able to find in the literature, we include the proof.
\begin{lemma}\label{lem_a:gronwall}
Let $0\leq t_0<\tau\leq\infty$, $c\geq 0$, $h\in L^p_{\rm loc}(\R_+,\R)$ be non-negative and
$y\in L^q_{\rm loc}(\R_+,\R)$ , where $1\leq p,q\leq\infty$ are such that $1/p+1/q=1$. If
\[
y(t)\leq c+\int_{t_0}^t h(s)y(s){\rm d}s,\quad\mbox{for a.e. $t\in[t_0,\tau)$},
\]
then $y(t)\leq ce^{\int_{t_0}^t h(s){\rm d}s}$ for almost every $t\in[t_0,\tau)$.
\end{lemma}
\begin{proof}
Setting $Y(t):=c+\int_{t_0}^t h(s)y(s){\rm d}s$ for all $t\in[t_0,\tau)$,
we have that $\dot Y(t)=h(t)y(t)\leq h(t)Y(t)$ for almost every $t\in[t_0,\tau)$, and so,
\[
\dot Y(t)-h(t)Y(t)\leq 0, \quad\mbox{for a.e. $t\in[t_0,\tau)$}.
\]
The function $w:[t_0,\tau)\to\R_+$ given by
\[
w(t)=Y(t)e^{-\int_{t_0}^t h(s){\rm d}s}\quad\forall\:t\in[t_0,\tau)
\]
is absolutely continuous and
\[
\dot w(t)=\big(\dot Y(t)-h(t)Y(t)\big)e^{-\int_{t_0}^t h(s){\rm d}s}\leq 0,
\quad\mbox{for a.e. $t\in[t_0,\tau)$},
\]
showing that $w$ is non-increasing. Consequently,
\[
Y(t)e^{-\int_{t_0}^t h(s){\rm d}s}=w(t)\leq w(t_0)=Y(t_0)=c,
\quad\forall\:t\in[t_0,\tau),
\]
and thus,
\[
y(t)\leq Y(t)\leq ce^{\int_{t_0}^t h(s){\rm d}s},\quad\mbox{for a.e. $t\in[t_0,\tau)$},
\]
completing the proof.
\end{proof}

\end{document}